\documentclass{amsart}

\usepackage{amsfonts,amsthm,amsmath,tikz-cd,tikz,graphicx}

\usetikzlibrary{shapes}
\usepackage[export]{adjustbox}
\usepackage{calc}

\usepackage[hidelinks]{hyperref}
\usepackage[dvipsnames]{xcolor}
\usepackage{libertine}
\usepackage[shortlabels]{enumitem}

\usepackage{transparent} 

\newtheorem{theorem}{Theorem}[section]

\newtheorem{lemma}[theorem]{Lemma}

\theoremstyle{remark}

\theoremstyle{definition}

\newtheorem*{notation}{Notation}
\newtheorem*{example}{Example}

\newcommand{\Sq}{\mathcal{S}_q}

\newcommand{\centerv}
[2] 
[0cm] 
{
 \raisebox{-0.5\height + #1} {#2}
}

\title{Steinberg skein identities at roots of unity}
\date{}

\author[Higgins]{Vijay Higgins}
\address{Department of Mathematics, University of California, Los Angeles, CA 90095, USA}
\email{higginsv@math.ucla.edu}

\author[Tambe]{Indraneel Tambe}
\address{Department of Mathematics, University of California, Los Angeles, CA 90095, USA}
\email{indraneel@math.ucla.edu}

\thanks{
	2020 {\em Mathematics Classification:} Primary 57K31. Secondary 17B37.}

\begin{document}

\begin{abstract}
We obtain a family of skein identities in the Kauffman bracket skein module which relate Frobenius elements to Jones-Wenzl projectors at roots of unity. We view these skein identities as certain incarnations of Steinberg tensor product formulae from the theory of tilting modules of the quantum group $U_q(sl_2)$. We show that the simplest skein identities yield a short new proof of the existence of the Chebyshev-Frobenius homomorphism of Bonahon-Wong.
\end{abstract}

\maketitle

\section{Introduction}
The Kauffman bracket skein module $\mathcal{S}_q(M)$ of an oriented $3\text{-manifold}$ $M$ was introduced by Przytycki and Turaev \cite{Prz91,Tur88} as a generalization of the Jones polynomial to links in $M.$ The skein module is spanned by unoriented framed links in $M$ modulo the Kauffman bracket skein relations. These skein relations are the same ones defining the Temperley-Lieb ribbon category which diagrammatically encodes the subcategory of representations of the quantum group $U_q(sl_2)$ monoidally generated by the defining $2\text{-dimensional}$ representation $V.$

When $q$ is not a root of unity, the whole representation category of $U_q(sl_2)$ may be obtained by taking the idempotent completion of the Temperley-Lieb category, with the $k+1$ dimensional irreducible representation $V_k=\text{Sym}^k(V)$ corresponding to $JW_k,$ the Jones-Wenzl idempotent on $k$ strands. On the other hand, when $q$ is a root of unity, idempotent completion of the Temperley-Lieb category yields the category of tilting modules of $U_q(sl_2),$ which are those modules that appear as direct summands of some tensor power $V^{\otimes k}$ of the defining representation \cite{Eli15}. In this way, one might view the Kauffman bracket skein module at a root of unity as being governed by the skein theory of tilting modules of $U_q(sl_2).$ However, one of the most surprising and important tools in studying skein modules at roots of unity, the Chebyshev-Frobenius homomorphism of Bonahon-Wong \cite{BW16}, has been come to be understood as involving strands in the Kauffman bracket skein module labeled by a certain non-tilting module, the image $Fr(V)$ of the standard representation $V$ under the so-called Frobenius functor of Lusztig \cite{Lus90}.

Our goal is to gain some understanding of the relationship between the Frobenius elements appearing in the skein module and the category of tilting modules governing the Temperley-Lieb category. To this end, we obtain skein identities involving both Frobenius elements and Jones-Wenzl projectors at roots of unity. We view the identities as certain incarnations of Steinberg tensor product formulae of Andersen-Wen \cite{AW92}. As an application, we show that the easiest skein identities yield a new proof of the existence of the Chebyshev-Frobenius homomorphism of Bonahon-Wong. All of our proofs use elementary and straightforward skein theoretic techniques which we estimate will generalize to skein theories for $\mathrm{SL}_d$ for $d>2$ to find alternate ways other than those of \cite{Hig25,KLW25} to study higher rank $\mathrm{SL}_d$ Frobenius homomorphisms. We hope the more complicated identities we prove can be used in the future to help compute the image of the Frobenius homomorphism on a web rather than a knot. Additionally, following \cite{STWZ23} our skein identities may be useful for studying modified traces \cite{GKPM11} and TQFTs in the non-semisimple setting.

\subsection{Main results}

In our setting, $q\in \mathbb{C}$ is a root of unity and $n$ is the smallest positive integer such that $q^n \in \{-1,1\}.$ The Jones-Wenzl projectors $JW_k$ exist for $1 \leq k \leq n-1$, as is evident from the well-known recursive construction of Jones-Wenzl projectors (see Equation (\ref{JW recursion})) . The element $JW_n$ does not exist, but the elements $JW_{mn-1}$ exist for all $m \geq 1$ \cite{GW93,STWZ23,MS22}.

We also recall that if $t=q^{n^2},$ then the Frobenius homomorphism $\mathcal{S}_t(M) \rightarrow \mathcal{S}_q(M)$ of Bonahon-Wong \cite{BW16} sends a framed knot $K$ to the threading $K^{[T_n]}$ of the Chebyshev polynomial $T_n$ along the knot. In the setting of Muller skein algebras \cite{Mul16} or stated skein algebras \cite{Le18,CL19}, a version of the Frobenius homomorphism sends an arc $\alpha$ to its $n^{\text{th}}$ framed power $\alpha^{(n)}$ \cite{LP19,BL22,KQ19}. From the works \cite{KQ19,BL22,GJS25, BR22} one can view either $K^{[T_n]}$ or $\alpha^{(n)}$ as an element colored by the Frobenius representation $Fr(V)$ of $U_q(sl_2)$ at the root of unity $q.$ These observations inspire us to obtain the following skein identities which relate these Frobenius elements to Jones-Wenzl projectors at roots of unity.

\begin{theorem}\label{intro Steinberg skein identities}(Theorem \ref{Steinberg skein identities})
When $q$ is a root of unity and $n$ is the smallest positive integer such that $q^n \in \{-1,1\},$ the following local skein identities hold in the Kauffman bracket skein module $\mathcal{S}_q(M)$ of any oriented $3\text{-manifold}$ $M.$

\begin{equation*}
\begin{tikzpicture}[baseline=6ex]
\draw [ultra thick] (0,0)--(0,2);
\node [draw, minimum width=.75cm, ultra thick, fill = white] (node) at (0,1) {$n-1$};
\draw [ultra thick] (.9,.25)--(.9,1.75);
\draw [ultra thick] (.9,1.75) arc (180:0:.25);
\draw [ultra thick] (1.4,1.75)--(1.4,.25);
\draw [ultra thick] (.9,.25) arc (-180:0:.25);
\filldraw [fill=white, draw=black] (1.15,1) circle [radius=.075];
\node (node) at (.9,1.75) [left] {$T_n$};
\end{tikzpicture}=
\begin{tikzpicture}[baseline=6ex]
\draw [ultra thick] (0,0)--(0,2);
\draw [ultra thick] (1.4,1.75) arc (0:90:.25);
\draw [ultra thick] (1.15,2) -- (.65,2);
\draw [ultra thick] (.65,2) arc (90:180:.25);
\draw [ultra thick] (.4,1.75)--(.4,.25);
\draw [ultra thick] (1.4,1.75)--(1.4,.25);
\draw [ultra thick] (1.4,.25) arc (0:-90:.25);
\draw [ultra thick] (1.15,0)--(.65,0);
\draw [ultra thick] (.65,0) arc (-90:-180:.25);
\node [draw, minimum width=.75cm, ultra thick, fill = white] (node) at (0,1) {$2n-1$};
\filldraw [fill=white, draw=black] (1.15,1) circle [radius=.075];
\node (node) at (.9,1.75) [right] {$n$};
\end{tikzpicture}
\text{ and }
\begin{tikzpicture}[baseline=6ex]
\draw [ultra thick] (0,0)--(0,2);
\node [draw, minimum width=.75cm, ultra thick, fill = white] (node) at (0,1) {$n-1$};
\draw [ultra thick] (.9,.25)--(.9,1.75);
\node (node) at (.9,1) [right] {$n$};
\node [draw, minimum width=.5cm, ultra thick, fill = white] (node) at (.9,1.75) {};
\node [draw, minimum width=.5cm, ultra thick, fill = white] (node) at (.9,.25) {};
\end{tikzpicture}=
\begin{tikzpicture}[baseline=6ex]
\draw [ultra thick] (-.25,0)--(-.25,2);
\draw [ultra thick] (.9, 1.75)--(.25,1);
\draw [ultra thick] (.9,.25)--(.25,1);
\node [draw, minimum width=.75cm, ultra thick, fill = white] (node) at (0,1) {$2n-1$};
\node [draw, minimum width=.5cm, ultra thick, fill = white] (node) at (.9,1.75) {};
\node [draw, minimum width=.5cm, ultra thick, fill = white] (node) at (.9,.25) {};
\node (node) at (.55,1.4) [right] {$n$};
\node (node) at (.55,.6) [right] {$n$};
\end{tikzpicture}. 
\end{equation*}
\end{theorem}

Here, the blank boxes represent Jones-Wenzl projectors on an unspecified number of strands. The dot in the first identity represents a region in $M$ that might not be contractible. The knot labeled by $T_n$ represents the knot threaded by the Chebyshev polynomial $T_n(x)$; see Section \ref{KBSM}. We view the skein identities as incarnations of the following Steinberg tensor product formula of Andersen-Wen \cite{AW92}; see Section \ref{Quantum group root of unity}.

\begin{equation*}
V_{n-1} \otimes Fr(V) \cong V_{2n-1}.
\end{equation*}

After Bonahon and Wong originally discovered the Chebyshev-Frobenius homomorphism using their quantum trace map \cite{BW11,BW16}, L\^{e} found a skein-theoretic proof of its existence \cite{Le15}. The homomorphism was used to characterize the centers and much of the representation theory of skein algebras at roots of unity in \cite{BW16,FKBL19,GJS25}. We show that our skein identities provide a short alternative proof of the homomorphism's existence.

\begin{theorem}(Theorem \ref{Chebyshev-Frobenius})
Our skein identities of Theorem \ref{intro Steinberg skein identities} imply the existence of the Chebyshev-Frobenius homomorphism of \cite{BW16}, yielding a new proof.
\end{theorem}

We introduce other Frobenius elements $\widetilde{JW_k}$ in skein modules (see Section \ref{thick JW}), which we call thick Jones-Wenzl elements, using a certain local notation involving green strands which is modeled on the result of applying the Frobenius homomorphism to a neighborhood containing an ordinary Jones-Wenzl projector $JW_k.$ Let $\widehat{k}=n-1+kn.$ Using the elements $\widetilde{JW}_k$ we obtain formulas for the projectors $JW_{\widehat{k}}$ for $k>1.$ We then prove the following skein identities, which provide relationships between these Frobenius elements and the Jones-Wenzl projectors $JW_{n-1}$ and $JW_{\widehat{k}}.$

\begin{theorem}\label{Intro other Steinberg skein identities}(Theorem \ref{other Steinberg skein identities})
Suppose $q$ is a root of unity and $n$ is the smallest positive integer such that $q^{n} \in \{-1,1\}.$ Let $k \geq 1$ and set $\widehat{k}=n-1+kn.$ The following skein identities hold in $\Sq(M)$ for any oriented $3\text{-manifold}$ $M.$

\begin{equation*}
\begin{tikzpicture}[baseline=6ex]
\draw [ultra thick] (0,0)--(0,2);
\draw [ultra thick, color=ForestGreen] (1.4,0)--(1.4,2);
\draw [ultra thick, color=ForestGreen] (1.1,0)--(1.1,2);
\draw [ultra thick, color=ForestGreen] (1.7,0)--(1.7,2);
\node [draw, minimum width=.75cm, ultra thick, fill = white] (node) at (0,1) {$n-1$};
\node [draw=ForestGreen, ellipse, minimum width=.75cm, ultra thick, fill=white] (node) at (1.4,1) {$k$};

\end{tikzpicture}=
\begin{tikzpicture}[baseline=6ex]
\draw [ultra thick] (-.25,0)--(-.25,2);
\draw [ultra thick, color=ForestGreen] (.05,0)--(.05,2);
\draw [ultra thick, color=ForestGreen] (.35,0)--(.35,2);
\draw [ultra thick, color=ForestGreen] (.65,0)--(.65,2);
\node [draw, minimum width=1.4cm, ultra thick, fill = white] (node) at (0,1) {$\widehat{k}$};
\node (node) at (-.25,2) [above] {$n-1$};
\node (node) at (-.25,0) [below] {$n-1$};
\end{tikzpicture}.
\end{equation*}
\end{theorem}

We view these skein identities as incarnations of the Steinberg tensor product formulas

\begin{equation*}
V_{n-1} \otimes Fr(V_{k}) \cong V_{n-1+kn}.
\end{equation*}

The $k=1$ case of Theorem \ref{Intro other Steinberg skein identities} is the same as Theorem \ref{intro Steinberg skein identities}. For each $k$ the identity in Theorem \ref{Intro other Steinberg skein identities} actually represents many skein identities, since the element $\widetilde{JW_k}$ can take on many forms in the skein module depending on the the destinations of the strands emanating from it.

\subsection{Acknowledgments}
V.H. and I.T. were partially supported by National Science Foundation RTG grant DMS-2136090. V.H. would like to thank Elijah Bodish for teaching him about the Steinberg tensor product formula.

\section{Quantum integers and Chebyshev polynomials}

For $q \in \mathbb{C}\setminus\{0\}$ and $k \in \mathbb{Z}_{\geq 0},$ the quantum integer $[k] \in \mathbb{C}$ is given by the expression $[k]=q^{k-1}+q^{k-3}+\cdots +q^{3-k}+q^{1-k}.$ If $q \neq \pm 1,$ we may write $[k]=(q^{k}-q^{-k})/(q-q^{-1}).$ The quantum factorial $[k]!$ is given by $[k]!=[k][k-1]\cdots[1].$

The following Chebyshev polynomials $T_k \in \mathbb{Z}[x]$ of the first kind and $S_k\in \mathbb{Z}[x]$ of the second kind are ubiquitious in representation theory and in skein theory. They are defined by the following recursive definitions, which show that they have integer coefficients.

For $k \geq 2,$ the polynomial $T_k$ is defined by the recursive formula

\begin{equation*}
T_k=xT_{k-1}-T_{k-2},
\end{equation*}
with the initial values of $T_1=x$ and $T_0=2.$

For $k \geq 2,$ the polynomial $S_k$ is given by the same recursive formula

\begin{equation*}
S_k=xS_{k-1}-S_{k-2},
\end{equation*}
but with initial values of $S_1=x$ and $S_0=1.$

They satisfy the properties

\begin{equation}\label{Tk defining property}
T_k(q+q^{-1})=q^k+q^{-k} \text{ and } S_k(q+q^{-1})=[k+1].
\end{equation}

We will make use of the identity

\begin{equation}\label{Tk and Sk}
T_k=xS_{k-1}-2S_{k-2}.
\end{equation}

\section{Kauffman bracket skein module and polynomial threadings}\label{KBSM}

The Kauffman bracket skein module $\mathcal{S}_q(M)$ of an oriented $3\text{-manifold } M$ is the module spanned by formal $\mathbb{C}\text{-linear}$ combinations of isotopy classes of framed unoriented links in $M$ modulo the following Kauffman bracket skein relations (\ref{Kauffman bracket relations}). Although the module depends on the choice of $q^{1/2},$ we will suppress this choice in the notation.

\begin{align}\label{Kauffman bracket relations}
\begin{tikzpicture}[baseline=1.5ex]
\draw [ultra thick] (0,0)--(1/2,1/2);
\node (a) at (1/4,1/4) {};
\draw [ultra thick] (0,1/2)--(a);
\draw [ultra thick] (a)--(1/2,0);
\end{tikzpicture}&=q^{1/2}
\begin{tikzpicture}[baseline=1.5ex]
\draw [ultra thick] (0,0) to [out=45,in=-45] (0,1/2);
\draw [ultra thick] (1/2,0) to [out=135, in=-135] (1/2,1/2);
\end{tikzpicture}+q^{-1/2}
\begin{tikzpicture}[baseline=1.5ex]
\draw [ultra thick] (0,1/2) to [out=-45, in =-135] (1/2,1/2);
\draw [ultra thick] (0,0) to [out=45, in=135] (1/2,0);
\end{tikzpicture}&
\begin{tikzpicture}[baseline=-.5ex]
\draw [ultra thick] (0,0) circle [radius=.25];
\end{tikzpicture}=-(q+q^{-1}).
\end{align}

When $M=\Sigma \times (0,1)$ is the thickening of an oriented surface $\Sigma$, the skein module $\mathcal{S}_q(M)$ has a natural algebra structure denoted by $\mathcal{S}_q(\Sigma)$ which is induced by the following product of links. If $L_1, L_2$ are two links in $\Sigma \times (0,1),$ their product is given by isotoping $L_1$ to obtain $L_1'$ contained in $\Sigma \times (1/2,1),$ isotoping $L_2$ to obtain $L_2'$ in $\Sigma \times (0,1/2)$ and then taking the union of the links $[L_1][L_2]=[L_1' \cup L_2'] \in \Sq(\Sigma).$

Here, we consider surfaces $\Sigma$ which have been obtained from a compact oriented surface by removing some nonnegative number of points, which we will refer to as punctures. One may study $\Sq(\Sigma)$ by working with link diagrams on the surface $\Sigma.$ By using the Kauffman bracket relations to remove crossings and trivial loops, one may rewrite any link diagram in $\Sq(\Sigma)$ in the basis consisting of multicurves on $\Sigma$ without trivial loops \cite{Prz91,SW07}.

When $\Sigma=A$ is the annulus, we see that $\Sq(A)=\mathbb{C}[x]$, where $x$ is the diagram of the core loop of the annulus. Any framed knot $K$ in an oriented manifold $M$ is given by an embedding $k: A \hookrightarrow M$. By the functoriality of skein modules \cite{Prz91}, $k$ induces a homomorphism of skein modules $k_*: \Sq(A) \rightarrow \Sq(M)$ which sends the core loop of the annulus to the element $[K] \in \Sq(M).$ For any polynomial $P \in \mathbb{C}[x]$, the element denoted by $K^{[P]} \in \Sq(M)$ is called the \underline{threading of $P$ along $K$} and is given by $K^{[P]}=k_*(P).$ If a link $L=\cup K_i$ is a union of knots $K_i,$ then the threading of $P$ along $L$ is given by the union of threadings of $P$ along each component $K_i$ of $L.$

Although the following lemma is a consequence of a deep and general result that the skein algebra of any surface has no nontrivial zero divisors \cite{BW11,PS19}, we will present a short elementary proof for our specific case.

\begin{lemma}\label{Steinberg loop}
Suppose $\Sigma$ is a surface with at least one puncture. Suppose that $P=\displaystyle\sum_{i=0}^{\text{deg}(P)}a_ix^i \in \mathbb{C}[x]$ is a nonzero polynomial and $[K]\in \Sq(\Sigma)$ is represented by a knot $K$ with a diagram that is a simple loop around a puncture of $\Sigma.$ If $K$ does not bound a disk on $\Sigma$ then $K^{[P]}$ is not a zero divisor in $\Sq(\Sigma).$
\end{lemma}

\begin{proof}
The proof strategy is to consider multiplication in the multicurve basis, keeping track of highest degree terms in the sense of the number of copies of $K$ in each multicurve. The details are as follows. For a basis multicurve diagram $D \in \Sq(\Sigma),$ let $m(D) \in \mathbb{Z}_{\geq 0}$ denote the multiplicity of $K$ in $D,$ which is the number of copies of the curves isotopic to $K$ appearing in $D.$ Then each basis diagram $D$ can be rewritten as $D=K^{m(D)}D' \in \Sq(\Sigma)$ where $D'$ is a basis diagram such that $m(D')=0.$ We have that two basis diagrams $D,E$ are isotopic if and only if $m(D)=m(E)$ and $E'$ is isotopic to $D'.$ Furthermore, we can extend the notion of the multiplicity of $K$ to consider arbitrary elements of the skein algebra. If $y=\sum b_iD_i$ is any element written as a linear combination of basis diagrams $D_i,$ then we define $m(y)=\displaystyle\max_{b_i \neq 0} \{m(D_i)\}.$ We define the top term of $y$ to be equal to $y_{\text{top}}=\displaystyle\sum_{m(D_i)=m(y)}b_iD_i.$

To show that $K^{[P]}$ is not a zero divisor, we must show that for any nonzero element $y \in \Sq(\Sigma)$ we have $K^{[P]}y \neq 0$ and $yK^{[P]} \neq 0.$ Since $K$ is a simple loop around a puncture on $\Sigma,$ $K^{[P]}$ is in the center of $\Sq(\Sigma)$. Thus, we only need to show the first non-equality. We will do this by showing that $(K^{[P]}y)_{\text{top}}\neq0.$

We suppose $y\neq 0$ and we write $y=\sum b_i D_i$ in the basis, where $b_i \in \mathbb{C}$ and $D_i$ are basis diagrams. We rewrite the sum by grouping the diagrams according to $m(D_i)$ in the following way

\begin{equation*}
y=\sum b_iD_i=\sum b_iK^{m(D_i)}D_i'=\sum_{l=0}^{m(y)}K^l\left(\sum_{m(D_i)=l} b_i D_i'\right).
\end{equation*}

We have that $y_{\text{top}}=K^{m(y)}\left(\displaystyle\sum_{m(D_i)=m(y)} b_i D_i'\right)\neq 0$ by assumption. We then compute that for $K^{[P]}=\displaystyle\sum_{j=0}^{\deg(P)} a_j K^j$ we have

\begin{equation*}
K^{[P]}y=\sum_{i,j} a_jb_iK^{m(D_i)+j}D_i'=\sum_{j=1}^{\deg(P)}\sum_{l=0}^{m(y)} a_jK^{l+j}\left(\sum_{m(D_i)=l} b_i D_i'\right).
\end{equation*}

We have that

\begin{equation}
(K^{[P]}y)_{\text{top}}=a_{\deg(P)}K^{\deg(P)+m(y)}\left(\sum_{m(D_i)=m(y)}b_iD_i'\right) \neq 0
\end{equation}
since it is a nontrivial linear combination of a nonempty set of distinct basis diagrams. Therefore, $K^{[P]}y \neq 0.$
\end{proof}

\section{Temperley-Lieb category and Jones-Wenzl projectors}

The Temperley-Lieb category $\text{TL}$ is built from certain tangle diagrams in a rectangle. $\mathrm{TL}$ has objects $n \in \mathbb{Z}_{\geq 0}.$ For objects $k,l,$ the vector space of morphisms $k \rightarrow l$ is given by $\mathrm{TL}_{k,l}$ and is spanned by diagrams of framed $(k,l)\text{-tangles}$ in a rectangle modulo the Kauffman bracket relations (\ref{Kauffman bracket relations}), with $k$ endpoints at the bottom of the rectangle and $l$ endpoints at the top. Composition in the category is given by concatenating diagrams. Although the tangle diagrams may contain crossings or trivial loops, the Kauffman bracket relations can be used to rewrite any diagram in $\mathrm{TL}_{k,l}$ as a linear combination of diagrams in the standard basis of planar matchings. For any object $k,$ the space $\mathrm{TL}_{k,k}$ is an endomorphism algebra where the identity element $\text{Id}_k \in \mathrm{TL}_{k,k}$ is given by the diagram of $k$ vertical strands.

Let $\text{Rep}_q(\mathrm{SL}_2)$ denote the category of finite dimensional representations of the quantum group $U_q(sl_2).$ The category $\text{TL}$ is isomorphic to the full subcategory of $\text{Rep}_q(\mathrm{SL}_2)$ whose objects are given by tensor powers $V^{\otimes k}$ of the standard 2-dimensional vector representation $V$; see \cite{RTW32,Kup96} for the classical and semisimple cases and \cite{Eli15} for the non-semisimple case. Consequently, the idempotent completion of $\mathrm{TL}$ yields the category of tilting modules of $\text{Rep}_q(\mathrm{SL}_2),$ consisting of objects which are direct summands of some $V^{\otimes k}.$ When $q=\pm 1$ or when $q$ is not a root of unity, this category is semisimple and every representation is a tilting module. However, when $q$ is a root of unity, there are interesting modules which are not tilting modules including the Frobenius module $Fr(V)$ which we will discuss in Section \ref{Quantum group root of unity}. 

When $q$ is generic, the correspondence between irreducible representations $V_k=\text{Sym}^k(V)$ and special elements of $\mathrm{TL}$ is given by the Jones-Wenzl projectors which we now recall. An element $JW_k$ of $\mathrm{TL}_{k,k}$ is called a \underline{Jones-Wenzl projector} if it satisfies the following two axiomatic properties:

\begin{enumerate}[i)]
    \item \label{axiom 1} The coefficient of $\text{Id}_k$ is $1$ when $JW_k$ is expressed in the standard basis of $\mathrm{TL}_{k,k},$
    \item \label{axiom 2} the element $JW_k$ is ``uncappable", meaning it satisfies (\ref{uncappable}):
\end{enumerate}

\begin{equation}\label{uncappable}
	\begin{tikzpicture}[baseline=6ex]
		\draw [ultra thick] (-0.5,.25)--(-0.5,1.75);
		\draw [ultra thick] (.5,.25)--(.5,1.75);
		\draw [dotted, ultra thick] (-.4,1.5)--(.4,1.5);
		\draw [ultra thick] (.25,1)--(.25,.6);
		\draw [ultra thick] (-.25,1)--(-.25,.6);
		\draw [ultra thick] (.25,.6) arc (0: -180: .25);
		\node [draw, minimum width=1.25cm, ultra thick, fill = white] (node) at (0,1) {$k$};
	\end{tikzpicture}=0=
	\begin{tikzpicture}[baseline=6ex]
		\draw [ultra thick] (-0.5,.25)--(-0.5,1.75);
		\draw [ultra thick] (.5,.25)--(.5,1.75);
		\draw [dotted, ultra thick] (-.4,.5)--(.4,.5);
		\draw [ultra thick] (.25,1)--(.25,1.4);
		\draw [ultra thick] (-.25,1)--(-.25,1.4);
		\draw [ultra thick] (.25,1.4) arc (0: 180: .25);
		\node [draw, minimum width=1.25cm, ultra thick, fill = white] (node) at (0,1) {$k$};
	\end{tikzpicture}
	.\end{equation}

Here, we draw the element $JW_k$ as a box labeled by $k.$ Each strand drawn is understood to depict a collection of a certain number of parallel strands, which may be indicated by an integer label next to the strand. By convention, a diagram is zero if any strand label is negative. We will often leave strands unlabeled whenever the label either is arbitrary or when it can be deduced from the diagram. The axioms guarantee that if $JW_k$ exists, it is unique. Further, they imply the following properties of absorption and both vertical and horizontal symmetry.

\begin{equation}\label{absorption}
\begin{tikzpicture}[baseline=6ex]
\draw [ultra thick] (0,0)--(0,1.75);
\draw [ultra thick] (-.5, .4)--(-.5,1.75);
\draw [ultra thick] (.5,.4)--(.5,1.75);
\node [draw, minimum width=1.25cm, ultra thick, fill = white] (node) at (0,.5) {$a$};
\node [draw, minimum width=.75cm, ultra thick, fill=white] (node) at (0,1.25) {$b$};
\end{tikzpicture}=
\begin{tikzpicture}[baseline=6ex]
\draw [ultra thick] (0,0)--(0,1.75);
\node [draw, minimum width=1.25cm, ultra thick, fill = white] (node) at (0,1) {$a$};
\end{tikzpicture} \text{ for } b \leq a, \text{ and }
\begin{tikzpicture}[baseline=6ex]
\draw [ultra thick] (0,0)--(0,2);
\node [draw, minimum width=1.25cm, ultra thick, fill = white] (node) at (0,1) {$\scalebox{-1}[1]{B}$};
\end{tikzpicture}=
\begin{tikzpicture}[baseline=6ex]
\draw [ultra thick] (0,0)--(0,2);
\node [draw, minimum width=1.25cm, ultra thick, fill = white] (node) at (0,1) {B};
\end{tikzpicture}=
\begin{tikzpicture}[baseline=6ex]
\draw [ultra thick] (0,0)--(0,2);
\node [draw, minimum width=1.25cm, ultra thick, fill = white] (node) at (0,1) {$\scalebox{1}[-1]{B}$};
\end{tikzpicture}.
\end{equation}

The following recursive definition of Jones-Wenzl projectors guarantees that each $JW_k$ exists if $[k]! \neq 0.$

\begin{equation}\label{JW recursion}
\begin{tikzpicture}[baseline=6ex]
\draw [ultra thick] (0,0)--(0,2);
\node [draw, minimum width=1.25cm, ultra thick, fill = white] (node) at (0,1) {$k$};
\end{tikzpicture}=
\begin{tikzpicture}[baseline=6ex]
\draw [ultra thick] (0,0)--(0,2);
\draw [ultra thick] (1,0)--(1,2);
\node [draw, minimum width=1.25cm, ultra thick, fill = white] (node) at (0,1) {$k-1$};
\end{tikzpicture}+\frac{[k-1]}{[k]}
\begin{tikzpicture}[baseline=6ex]
\draw [ultra thick] (-0.5,0.5)--(-0.5,1.5);
\draw [ultra thick] (0,0)--(0,.6);
\draw [ultra thick] (0,2)--(0,1.4);
\draw [ultra thick] (1,1.4) arc (0: -180: .25);
\draw [ultra thick] (1,.6) arc (0: 180: .25);
\draw [ultra thick] (1,0)--(1,.6);
\draw [ultra thick] (1,2)--(1,1.4);
\node [draw, minimum width=1.25cm, ultra thick, fill = white] (node) at (0,1.5) {$k-1$};
\node [draw, minimum width=1.25cm, ultra thick, fill = white] (node) at (0,.5) {$k-1$};
\end{tikzpicture}.
\end{equation}

We will make use of the following well-known identities.

\begin{lemma}
Assume $JW_k$ exists for $1 \leq k \leq m.$ Then we have 

\begin{equation}\label{JW trace}
\centerv {
\def\svgscale{0.6}
\begingroup%
  \makeatletter%
  \providecommand\color[2][]{%
    \errmessage{(Inkscape) Color is used for the text in Inkscape, but the package 'color.sty' is not loaded}%
    \renewcommand\color[2][]{}%
  }%
  \providecommand\transparent[1]{%
    \errmessage{(Inkscape) Transparency is used (non-zero) for the text in Inkscape, but the package 'transparent.sty' is not loaded}%
    \renewcommand\transparent[1]{}%
  }%
  \providecommand\rotatebox[2]{#2}%
  \newcommand*\fsize{\dimexpr\f@size pt\relax}%
  \newcommand*\lineheight[1]{\fontsize{\fsize}{#1\fsize}\selectfont}%
  \ifx\svgwidth\undefined%
    \setlength{\unitlength}{88.88556353bp}%
    \ifx\svgscale\undefined%
      \relax%
    \else%
      \setlength{\unitlength}{\unitlength * \real{\svgscale}}%
    \fi%
  \else%
    \setlength{\unitlength}{\svgwidth}%
  \fi%
  \global\let\svgwidth\undefined%
  \global\let\svgscale\undefined%
  \makeatother%
  \begin{picture}(1,0.95672883)%
    \lineheight{1}%
    \setlength\tabcolsep{0pt}%
    \put(0,0){\includegraphics[width=\unitlength,page=1]{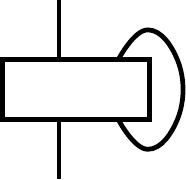}}%
    \put(0.20400246,0.43139691){\color[rgb]{0,0,0}\makebox(0,0)[lt]{\lineheight{1.25}\smash{\begin{tabular}[t]{l}$\scriptstyle m$\end{tabular}}}}%
    \put(0.79212653,0.85727666){\color[rgb]{0,0,0}\makebox(0,0)[lt]{\lineheight{1.25}\smash{\begin{tabular}[t]{l}$\scriptstyle 1$\end{tabular}}}}%
  \end{picture}%
\endgroup%

}
= (-1)\frac{[m+1]}{[m]} \times
\centerv {
\def\svgscale{0.6}
\begingroup%
  \makeatletter%
  \providecommand\color[2][]{%
    \errmessage{(Inkscape) Color is used for the text in Inkscape, but the package 'color.sty' is not loaded}%
    \renewcommand\color[2][]{}%
  }%
  \providecommand\transparent[1]{%
    \errmessage{(Inkscape) Transparency is used (non-zero) for the text in Inkscape, but the package 'transparent.sty' is not loaded}%
    \renewcommand\transparent[1]{}%
  }%
  \providecommand\rotatebox[2]{#2}%
  \newcommand*\fsize{\dimexpr\f@size pt\relax}%
  \newcommand*\lineheight[1]{\fontsize{\fsize}{#1\fsize}\selectfont}%
  \ifx\svgwidth\undefined%
    \setlength{\unitlength}{57.56680226bp}%
    \ifx\svgscale\undefined%
      \relax%
    \else%
      \setlength{\unitlength}{\unitlength * \real{\svgscale}}%
    \fi%
  \else%
    \setlength{\unitlength}{\svgwidth}%
  \fi%
  \global\let\svgwidth\undefined%
  \global\let\svgscale\undefined%
  \makeatother%
  \begin{picture}(1,1.47722935)%
    \lineheight{1}%
    \setlength\tabcolsep{0pt}%
    \put(0,0){\includegraphics[width=\unitlength,page=1]{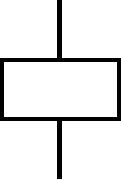}}%
    \put(0.13977526,0.68353869){\color[rgb]{0,0,0}\makebox(0,0)[lt]{\lineheight{1.25}\smash{\begin{tabular}[t]{l}$\scriptstyle m-1$\end{tabular}}}}%
  \end{picture}%
\endgroup%

}
.
\end{equation}

\begin{equation}\label{JW closed}
\centerv{\def\svgscale{0.5}
\begingroup%
  \makeatletter%
  \providecommand\color[2][]{%
    \errmessage{(Inkscape) Color is used for the text in Inkscape, but the package 'color.sty' is not loaded}%
    \renewcommand\color[2][]{}%
  }%
  \providecommand\transparent[1]{%
    \errmessage{(Inkscape) Transparency is used (non-zero) for the text in Inkscape, but the package 'transparent.sty' is not loaded}%
    \renewcommand\transparent[1]{}%
  }%
  \providecommand\rotatebox[2]{#2}%
  \newcommand*\fsize{\dimexpr\f@size pt\relax}%
  \newcommand*\lineheight[1]{\fontsize{\fsize}{#1\fsize}\selectfont}%
  \ifx\svgwidth\undefined%
    \setlength{\unitlength}{101.46328183bp}%
    \ifx\svgscale\undefined%
      \relax%
    \else%
      \setlength{\unitlength}{\unitlength * \real{\svgscale}}%
    \fi%
  \else%
    \setlength{\unitlength}{\svgwidth}%
  \fi%
  \global\let\svgwidth\undefined%
  \global\let\svgscale\undefined%
  \makeatother%
  \begin{picture}(1,0.7504815)%
    \lineheight{1}%
    \setlength\tabcolsep{0pt}%
    \put(0,0){\includegraphics[width=\unitlength,page=1]{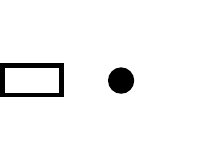}}%
    \put(0.07410974,0.34801418){\color[rgb]{0,0,0}\makebox(0,0)[lt]{\lineheight{1.25}\smash{\begin{tabular}[t]{l}$\scriptstyle m$\end{tabular}}}}%
    \put(0,0){\includegraphics[width=\unitlength,page=2]{2.0-JWcloseSm-start.pdf}}%
  \end{picture}%
\endgroup%

}
=
\centerv{\def\svgscale{0.5}
\begingroup%
  \makeatletter%
  \providecommand\color[2][]{%
    \errmessage{(Inkscape) Color is used for the text in Inkscape, but the package 'color.sty' is not loaded}%
    \renewcommand\color[2][]{}%
  }%
  \providecommand\transparent[1]{%
    \errmessage{(Inkscape) Transparency is used (non-zero) for the text in Inkscape, but the package 'transparent.sty' is not loaded}%
    \renewcommand\transparent[1]{}%
  }%
  \providecommand\rotatebox[2]{#2}%
  \newcommand*\fsize{\dimexpr\f@size pt\relax}%
  \newcommand*\lineheight[1]{\fontsize{\fsize}{#1\fsize}\selectfont}%
  \ifx\svgwidth\undefined%
    \setlength{\unitlength}{103.67101474bp}%
    \ifx\svgscale\undefined%
      \relax%
    \else%
      \setlength{\unitlength}{\unitlength * \real{\svgscale}}%
    \fi%
  \else%
    \setlength{\unitlength}{\svgwidth}%
  \fi%
  \global\let\svgwidth\undefined%
  \global\let\svgscale\undefined%
  \makeatother%
  \begin{picture}(1,0.77930529)%
    \lineheight{1}%
    \setlength\tabcolsep{0pt}%
    \put(0,0){\includegraphics[width=\unitlength,page=1]{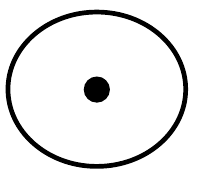}}%
    \put(0.76167114,0.66850081){\color[rgb]{0,0,0}\makebox(0,0)[lt]{\lineheight{1.25}\smash{\begin{tabular}[t]{l}$\scriptstyle S_m$\end{tabular}}}}%
  \end{picture}%
\endgroup%

}.
\end{equation}

\end{lemma}

\begin{proof}
These identities are consequences of the recursion (\ref{JW recursion}). Full proofs can be found in \cite{Lic97}.
\end{proof}

\begin{lemma}\label{crossing absorption}
Suppose $JW_M$ exists. Then for any $k,l \geq 0$ we have
\begin{equation}
\centerv{ \def\svgscale{0.5}
\begingroup%
  \makeatletter%
  \providecommand\color[2][]{%
    \errmessage{(Inkscape) Color is used for the text in Inkscape, but the package 'color.sty' is not loaded}%
    \renewcommand\color[2][]{}%
  }%
  \providecommand\transparent[1]{%
    \errmessage{(Inkscape) Transparency is used (non-zero) for the text in Inkscape, but the package 'transparent.sty' is not loaded}%
    \renewcommand\transparent[1]{}%
  }%
  \providecommand\rotatebox[2]{#2}%
  \newcommand*\fsize{\dimexpr\f@size pt\relax}%
  \newcommand*\lineheight[1]{\fontsize{\fsize}{#1\fsize}\selectfont}%
  \ifx\svgwidth\undefined%
    \setlength{\unitlength}{85.03936467bp}%
    \ifx\svgscale\undefined%
      \relax%
    \else%
      \setlength{\unitlength}{\unitlength * \real{\svgscale}}%
    \fi%
  \else%
    \setlength{\unitlength}{\svgwidth}%
  \fi%
  \global\let\svgwidth\undefined%
  \global\let\svgscale\undefined%
  \makeatother%
  \begin{picture}(1,1.19313337)%
    \lineheight{1}%
    \setlength\tabcolsep{0pt}%
    \put(0,0){\includegraphics[width=\unitlength,page=1]{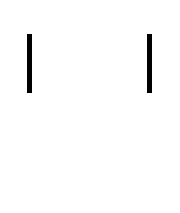}}%
    \put(0.3070526,1.05704965){\color[rgb]{0,0,0}\makebox(0,0)[lt]{\lineheight{1.25}\smash{\begin{tabular}[t]{l}$\scriptstyle k$\end{tabular}}}}%
    \put(0.60028331,1.06471872){\color[rgb]{0,0,0}\makebox(0,0)[lt]{\lineheight{1.25}\smash{\begin{tabular}[t]{l}$\scriptstyle \ell$\end{tabular}}}}%
    \put(0.4025477,0.54919185){\color[rgb]{0,0,0}\makebox(0,0)[lt]{\lineheight{1.25}\smash{\begin{tabular}[t]{l}$\scriptscriptstyle M$\end{tabular}}}}%
    \put(0,0){\includegraphics[width=\unitlength,page=2]{2.0-JWcross-startflip-opp.pdf}}%
  \end{picture}%
\endgroup%

}=q^{kl/2}\centerv{ \def\svgscale{0.6}
\begingroup%
  \makeatletter%
  \providecommand\color[2][]{%
    \errmessage{(Inkscape) Color is used for the text in Inkscape, but the package 'color.sty' is not loaded}%
    \renewcommand\color[2][]{}%
  }%
  \providecommand\transparent[1]{%
    \errmessage{(Inkscape) Transparency is used (non-zero) for the text in Inkscape, but the package 'transparent.sty' is not loaded}%
    \renewcommand\transparent[1]{}%
  }%
  \providecommand\rotatebox[2]{#2}%
  \newcommand*\fsize{\dimexpr\f@size pt\relax}%
  \newcommand*\lineheight[1]{\fontsize{\fsize}{#1\fsize}\selectfont}%
  \ifx\svgwidth\undefined%
    \setlength{\unitlength}{30.59716172bp}%
    \ifx\svgscale\undefined%
      \relax%
    \else%
      \setlength{\unitlength}{\unitlength * \real{\svgscale}}%
    \fi%
  \else%
    \setlength{\unitlength}{\svgwidth}%
  \fi%
  \global\let\svgwidth\undefined%
  \global\let\svgscale\undefined%
  \makeatother%
  \begin{picture}(1,2.31610146)%
    \lineheight{1}%
    \setlength\tabcolsep{0pt}%
    \put(0.22403492,1.06283569){\color[rgb]{0,0,0}\makebox(0,0)[lt]{\lineheight{1.25}\smash{\begin{tabular}[t]{l}$\scriptstyle M$\end{tabular}}}}%
    \put(0,0){\includegraphics[width=\unitlength,page=1]{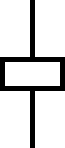}}%
  \end{picture}%
\endgroup%

}.
\end{equation}
\end{lemma}

\begin{proof}
Resolve the crossing closest to the $JW$ projector using (\ref{Kauffman bracket relations}) and then use the uncappable property (\ref{uncappable}) of the $JW$ projector. Repeat this step until all crossings are gone. In this way, each crossing accounts for a factor of $q^{1/2}$ for an overall factor of $(q^{1/2})^{kl}.$
\end{proof}

In the following lemma and later, a blank box represents a $JW$ projector on an unspecified number of strands at least as large as the number of strands drawn. Although they may exist, we do not draw strands leaving the blank box that are not involved in the identity.

\begin{lemma}\label{JW crossing}
\begin{equation}
\centerv { \def\svgscale{0.5}
\begingroup%
  \makeatletter%
  \providecommand\color[2][]{%
    \errmessage{(Inkscape) Color is used for the text in Inkscape, but the package 'color.sty' is not loaded}%
    \renewcommand\color[2][]{}%
  }%
  \providecommand\transparent[1]{%
    \errmessage{(Inkscape) Transparency is used (non-zero) for the text in Inkscape, but the package 'transparent.sty' is not loaded}%
    \renewcommand\transparent[1]{}%
  }%
  \providecommand\rotatebox[2]{#2}%
  \newcommand*\fsize{\dimexpr\f@size pt\relax}%
  \newcommand*\lineheight[1]{\fontsize{\fsize}{#1\fsize}\selectfont}%
  \ifx\svgwidth\undefined%
    \setlength{\unitlength}{157.23641463bp}%
    \ifx\svgscale\undefined%
      \relax%
    \else%
      \setlength{\unitlength}{\unitlength * \real{\svgscale}}%
    \fi%
  \else%
    \setlength{\unitlength}{\svgwidth}%
  \fi%
  \global\let\svgwidth\undefined%
  \global\let\svgscale\undefined%
  \makeatother%
  \begin{picture}(1,0.90232879)%
    \lineheight{1}%
    \setlength\tabcolsep{0pt}%
    \put(0,0){\includegraphics[width=\unitlength,page=1]{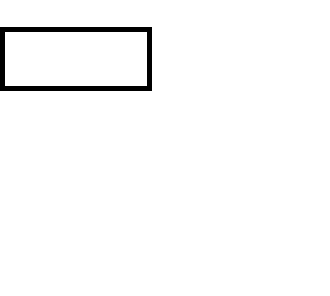}}%
    \put(0.65032765,0.37691636){\color[rgb]{0,0,0}\makebox(0,0)[lt]{\lineheight{1.25}\smash{\begin{tabular}[t]{l}$\scriptstyle m$\end{tabular}}}}%
    \put(0,0){\includegraphics[width=\unitlength,page=2]{2.4-ncross-prelim1-start.pdf}}%
    \put(0.10948995,0.10706463){\color[rgb]{0,0,0}\makebox(0,0)[lt]{\lineheight{1.25}\smash{\begin{tabular}[t]{l}$\scriptstyle 1$\end{tabular}}}}%
  \end{picture}%
\endgroup%

}
=
q^{m/2}
\centerv { \def\svgscale{0.5}
\begingroup%
  \makeatletter%
  \providecommand\color[2][]{%
    \errmessage{(Inkscape) Color is used for the text in Inkscape, but the package 'color.sty' is not loaded}%
    \renewcommand\color[2][]{}%
  }%
  \providecommand\transparent[1]{%
    \errmessage{(Inkscape) Transparency is used (non-zero) for the text in Inkscape, but the package 'transparent.sty' is not loaded}%
    \renewcommand\transparent[1]{}%
  }%
  \providecommand\rotatebox[2]{#2}%
  \newcommand*\fsize{\dimexpr\f@size pt\relax}%
  \newcommand*\lineheight[1]{\fontsize{\fsize}{#1\fsize}\selectfont}%
  \ifx\svgwidth\undefined%
    \setlength{\unitlength}{157.23642544bp}%
    \ifx\svgscale\undefined%
      \relax%
    \else%
      \setlength{\unitlength}{\unitlength * \real{\svgscale}}%
    \fi%
  \else%
    \setlength{\unitlength}{\svgwidth}%
  \fi%
  \global\let\svgwidth\undefined%
  \global\let\svgscale\undefined%
  \makeatother%
  \begin{picture}(1,0.90139592)%
    \lineheight{1}%
    \setlength\tabcolsep{0pt}%
    \put(0,0){\includegraphics[width=\unitlength,page=1]{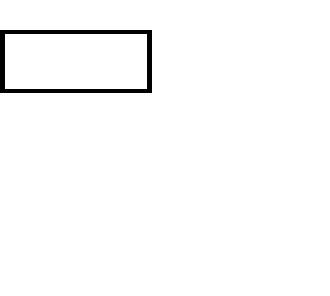}}%
    \put(0.41464221,0.5408375){\color[rgb]{0,0,0}\rotatebox{-34.879051}{\makebox(0,0)[lt]{\lineheight{1.25}\smash{\begin{tabular}[t]{l}$\scriptstyle m-1$\end{tabular}}}}}%
    \put(0,0){\includegraphics[width=\unitlength,page=2]{2.4-ncross-prelim1-end-term1.pdf}}%
    \put(0.01511818,0.28631031){\color[rgb]{0,0,0}\makebox(0,0)[lt]{\lineheight{1.25}\smash{\begin{tabular}[t]{l}$\scriptstyle 1$\end{tabular}}}}%
    \put(0.82928431,0.55639983){\color[rgb]{0,0,0}\makebox(0,0)[lt]{\lineheight{1.25}\smash{\begin{tabular}[t]{l}$\scriptstyle 1$\end{tabular}}}}%
    \put(0,0){\includegraphics[width=\unitlength,page=3]{2.4-ncross-prelim1-end-term1.pdf}}%
  \end{picture}%
\endgroup%

}
+
q^{-m/2}
\centerv { \def\svgscale{0.5}
\begingroup%
  \makeatletter%
  \providecommand\color[2][]{%
    \errmessage{(Inkscape) Color is used for the text in Inkscape, but the package 'color.sty' is not loaded}%
    \renewcommand\color[2][]{}%
  }%
  \providecommand\transparent[1]{%
    \errmessage{(Inkscape) Transparency is used (non-zero) for the text in Inkscape, but the package 'transparent.sty' is not loaded}%
    \renewcommand\transparent[1]{}%
  }%
  \providecommand\rotatebox[2]{#2}%
  \newcommand*\fsize{\dimexpr\f@size pt\relax}%
  \newcommand*\lineheight[1]{\fontsize{\fsize}{#1\fsize}\selectfont}%
  \ifx\svgwidth\undefined%
    \setlength{\unitlength}{157.23640382bp}%
    \ifx\svgscale\undefined%
      \relax%
    \else%
      \setlength{\unitlength}{\unitlength * \real{\svgscale}}%
    \fi%
  \else%
    \setlength{\unitlength}{\svgwidth}%
  \fi%
  \global\let\svgwidth\undefined%
  \global\let\svgscale\undefined%
  \makeatother%
  \begin{picture}(1,0.90635719)%
    \lineheight{1}%
    \setlength\tabcolsep{0pt}%
    \put(0,0){\includegraphics[width=\unitlength,page=1]{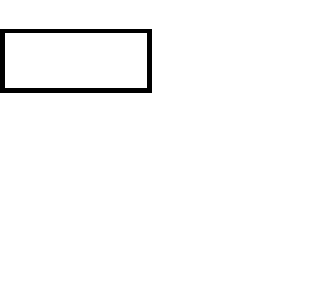}}%
    \put(0.12619551,0.54455652){\color[rgb]{0,0,0}\rotatebox{-31.232846}{\makebox(0,0)[lt]{\lineheight{1.25}\smash{\begin{tabular}[t]{l}$\scriptstyle m-1$\end{tabular}}}}}%
    \put(0,0){\includegraphics[width=\unitlength,page=2]{2.4-ncross-prelim1-step3-term2.pdf}}%
    \put(0.12619542,0.10901161){\color[rgb]{0,0,0}\makebox(0,0)[lt]{\lineheight{1.25}\smash{\begin{tabular}[t]{l}$\scriptstyle 1$\end{tabular}}}}%
    \put(0.64609546,0.63255968){\color[rgb]{0,0,0}\makebox(0,0)[lt]{\lineheight{1.25}\smash{\begin{tabular}[t]{l}$\scriptstyle 1$\end{tabular}}}}%
    \put(0,0){\includegraphics[width=\unitlength,page=3]{2.4-ncross-prelim1-step3-term2.pdf}}%
  \end{picture}%
\endgroup%

}.
\end{equation}    
\end{lemma}

\begin{proof}
Resolve the left-most crossing and then apply Lemma \ref{crossing absorption}.
\end{proof}

\begin{lemma}\label{triangle}
Suppose that $JW_k$ exists for $1 \leq k \leq n-1.$ Then the following identity holds for any $m \leq n-1.$
\begin{equation}
\centerv {
\def\svgscale{0.6}
\begingroup%
  \makeatletter%
  \providecommand\color[2][]{%
    \errmessage{(Inkscape) Color is used for the text in Inkscape, but the package 'color.sty' is not loaded}%
    \renewcommand\color[2][]{}%
  }%
  \providecommand\transparent[1]{%
    \errmessage{(Inkscape) Transparency is used (non-zero) for the text in Inkscape, but the package 'transparent.sty' is not loaded}%
    \renewcommand\transparent[1]{}%
  }%
  \providecommand\rotatebox[2]{#2}%
  \newcommand*\fsize{\dimexpr\f@size pt\relax}%
  \newcommand*\lineheight[1]{\fontsize{\fsize}{#1\fsize}\selectfont}%
  \ifx\svgwidth\undefined%
    \setlength{\unitlength}{81.6200755bp}%
    \ifx\svgscale\undefined%
      \relax%
    \else%
      \setlength{\unitlength}{\unitlength * \real{\svgscale}}%
    \fi%
  \else%
    \setlength{\unitlength}{\svgwidth}%
  \fi%
  \global\let\svgwidth\undefined%
  \global\let\svgscale\undefined%
  \makeatother%
  \begin{picture}(1,1.21554148)%
    \lineheight{1}%
    \setlength\tabcolsep{0pt}%
    \put(0,0){\includegraphics[width=\unitlength,page=1]{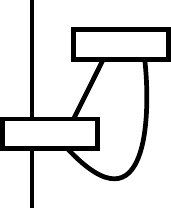}}%
    \put(0.53537723,0.62994808){\color[rgb]{0,0,0}\makebox(0,0)[lt]{\lineheight{1.25}\smash{\begin{tabular}[t]{l}$\scriptscriptstyle m$\end{tabular}}}}%
    \put(0.65298179,0.22950604){\color[rgb]{0,0,0}\makebox(0,0)[lt]{\lineheight{1.25}\smash{\begin{tabular}[t]{l}$\scriptscriptstyle 1$\end{tabular}}}}%
    \put(0.12340541,0.39440059){\color[rgb]{0,0,0}\makebox(0,0)[lt]{\lineheight{1.25}\smash{\begin{tabular}[t]{l}$\scriptscriptstyle n-1$\end{tabular}}}}%
  \end{picture}%
\endgroup%

}
=
\frac{[n-m-1]}{[n-1]} 
\centerv  {
\def\svgscale{0.6}
\begingroup%
  \makeatletter%
  \providecommand\color[2][]{%
    \errmessage{(Inkscape) Color is used for the text in Inkscape, but the package 'color.sty' is not loaded}%
    \renewcommand\color[2][]{}%
  }%
  \providecommand\transparent[1]{%
    \errmessage{(Inkscape) Transparency is used (non-zero) for the text in Inkscape, but the package 'transparent.sty' is not loaded}%
    \renewcommand\transparent[1]{}%
  }%
  \providecommand\rotatebox[2]{#2}%
  \newcommand*\fsize{\dimexpr\f@size pt\relax}%
  \newcommand*\lineheight[1]{\fontsize{\fsize}{#1\fsize}\selectfont}%
  \ifx\svgwidth\undefined%
    \setlength{\unitlength}{87.28937224bp}%
    \ifx\svgscale\undefined%
      \relax%
    \else%
      \setlength{\unitlength}{\unitlength * \real{\svgscale}}%
    \fi%
  \else%
    \setlength{\unitlength}{\svgwidth}%
  \fi%
  \global\let\svgwidth\undefined%
  \global\let\svgscale\undefined%
  \makeatother%
  \begin{picture}(1,1.29896498)%
    \lineheight{1}%
    \setlength\tabcolsep{0pt}%
    \put(0,0){\includegraphics[width=\unitlength,page=1]{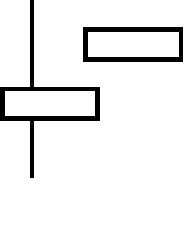}}%
    \put(0.84796925,0.5864876){\color[rgb]{0,0,0}\makebox(0,0)[lt]{\lineheight{1.25}\smash{\begin{tabular}[t]{l}$\scriptscriptstyle m$\end{tabular}}}}%
    \put(0.40405723,0.47309129){\color[rgb]{0,0,0}\makebox(0,0)[lt]{\lineheight{1.25}\smash{\begin{tabular}[t]{l}$\scriptscriptstyle 1$\end{tabular}}}}%
    \put(0.02395898,0.69418708){\color[rgb]{0,0,0}\makebox(0,0)[lt]{\lineheight{1.25}\smash{\begin{tabular}[t]{l}$\scriptscriptstyle n-m-1$\end{tabular}}}}%
    \put(0,0){\includegraphics[width=\unitlength,page=2]{2.1-JWrecursionlemma1-RHS.pdf}}%
    \put(0.11539043,0.20641462){\color[rgb]{0,0,0}\makebox(0,0)[lt]{\lineheight{1.25}\smash{\begin{tabular}[t]{l}$\scriptscriptstyle n-2$\end{tabular}}}}%
    \put(0,0){\includegraphics[width=\unitlength,page=3]{2.1-JWrecursionlemma1-RHS.pdf}}%
  \end{picture}%
\endgroup%

}.
\end{equation}
\end{lemma}

\begin{proof}
The identity can be proven by repeatedly applying the recursion (\ref{JW recursion}) to the spot occupied by $JW_{n-1}$ and using the uncappable property (\ref{uncappable}) of the blank $JW$ projector. Each time the recursion is applied, only one of the two terms survives. See \cite[Lemma 2]{MV94}  or \cite[Lemma 4.29]{STWZ23} for a full proof of a more general identity.
\end{proof}

The following sequel to Lemma \ref{triangle} will be used by us to create terms later involved in a telescoping sum in the proof of Theorem \ref{Steinberg skein identities}.
\begin{lemma}\label{telescoping}
Suppose that $JW_k$ exists for $1 \leq k \leq n-1.$ Then the following identity holds for any $m \leq n-2.$
\begin{equation}
\centerv {
\def\svgscale{0.6}
\begingroup%
  \makeatletter%
  \providecommand\color[2][]{%
    \errmessage{(Inkscape) Color is used for the text in Inkscape, but the package 'color.sty' is not loaded}%
    \renewcommand\color[2][]{}%
  }%
  \providecommand\transparent[1]{%
    \errmessage{(Inkscape) Transparency is used (non-zero) for the text in Inkscape, but the package 'transparent.sty' is not loaded}%
    \renewcommand\transparent[1]{}%
  }%
  \providecommand\rotatebox[2]{#2}%
  \newcommand*\fsize{\dimexpr\f@size pt\relax}%
  \newcommand*\lineheight[1]{\fontsize{\fsize}{#1\fsize}\selectfont}%
  \ifx\svgwidth\undefined%
    \setlength{\unitlength}{115.63581812bp}%
    \ifx\svgscale\undefined%
      \relax%
    \else%
      \setlength{\unitlength}{\unitlength * \real{\svgscale}}%
    \fi%
  \else%
    \setlength{\unitlength}{\svgwidth}%
  \fi%
  \global\let\svgwidth\undefined%
  \global\let\svgscale\undefined%
  \makeatother%
  \begin{picture}(1,0.85797463)%
    \lineheight{1}%
    \setlength\tabcolsep{0pt}%
    \put(0,0){\includegraphics[width=\unitlength,page=1]{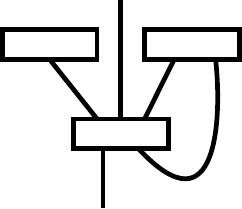}}%
    \put(0.67205195,0.44464093){\color[rgb]{0,0,0}\makebox(0,0)[lt]{\lineheight{1.25}\smash{\begin{tabular}[t]{l}$\scriptscriptstyle m$\end{tabular}}}}%
    \put(0.50849586,0.77870328){\color[rgb]{0,0,0}\makebox(0,0)[lt]{\lineheight{1.25}\smash{\begin{tabular}[t]{l}$\scriptscriptstyle 1$\end{tabular}}}}%
    \put(0.75506164,0.16199399){\color[rgb]{0,0,0}\makebox(0,0)[lt]{\lineheight{1.25}\smash{\begin{tabular}[t]{l}$\scriptscriptstyle 1$\end{tabular}}}}%
    \put(0.38126687,0.27838273){\color[rgb]{0,0,0}\makebox(0,0)[lt]{\lineheight{1.25}\smash{\begin{tabular}[t]{l}$\scriptscriptstyle n-1$\end{tabular}}}}%
  \end{picture}%
\endgroup%

}
=
\tfrac{[n-m-1]}{[n-1]}
\centerv {
\def\svgscale{0.6}
\begingroup%
  \makeatletter%
  \providecommand\color[2][]{%
    \errmessage{(Inkscape) Color is used for the text in Inkscape, but the package 'color.sty' is not loaded}%
    \renewcommand\color[2][]{}%
  }%
  \providecommand\transparent[1]{%
    \errmessage{(Inkscape) Transparency is used (non-zero) for the text in Inkscape, but the package 'transparent.sty' is not loaded}%
    \renewcommand\transparent[1]{}%
  }%
  \providecommand\rotatebox[2]{#2}%
  \newcommand*\fsize{\dimexpr\f@size pt\relax}%
  \newcommand*\lineheight[1]{\fontsize{\fsize}{#1\fsize}\selectfont}%
  \ifx\svgwidth\undefined%
    \setlength{\unitlength}{115.63581812bp}%
    \ifx\svgscale\undefined%
      \relax%
    \else%
      \setlength{\unitlength}{\unitlength * \real{\svgscale}}%
    \fi%
  \else%
    \setlength{\unitlength}{\svgwidth}%
  \fi%
  \global\let\svgwidth\undefined%
  \global\let\svgscale\undefined%
  \makeatother%
  \begin{picture}(1,0.85840815)%
    \lineheight{1}%
    \setlength\tabcolsep{0pt}%
    \put(0,0){\includegraphics[width=\unitlength,page=1]{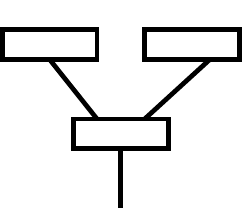}}%
    \put(0.51009211,0.78189604){\color[rgb]{0,0,0}\makebox(0,0)[lt]{\lineheight{1.25}\smash{\begin{tabular}[t]{l}$\scriptscriptstyle 1$\end{tabular}}}}%
    \put(0.38126683,0.2783828){\color[rgb]{0,0,0}\makebox(0,0)[lt]{\lineheight{1.25}\smash{\begin{tabular}[t]{l}$\scriptscriptstyle n-2$\end{tabular}}}}%
    \put(0,0){\includegraphics[width=\unitlength,page=2]{2.1-JWrecursionlemma2-RHS-term1.pdf}}%
    \put(0.75582966,0.44309077){\color[rgb]{0,0,0}\makebox(0,0)[lt]{\lineheight{1.25}\smash{\begin{tabular}[t]{l}$\scriptscriptstyle m$\end{tabular}}}}%
  \end{picture}%
\endgroup%

}
+
\tfrac{[n-m-2]}{[n-1]}
\centerv {
\def\svgscale{0.6}
\begingroup%
  \makeatletter%
  \providecommand\color[2][]{%
    \errmessage{(Inkscape) Color is used for the text in Inkscape, but the package 'color.sty' is not loaded}%
    \renewcommand\color[2][]{}%
  }%
  \providecommand\transparent[1]{%
    \errmessage{(Inkscape) Transparency is used (non-zero) for the text in Inkscape, but the package 'transparent.sty' is not loaded}%
    \renewcommand\transparent[1]{}%
  }%
  \providecommand\rotatebox[2]{#2}%
  \newcommand*\fsize{\dimexpr\f@size pt\relax}%
  \newcommand*\lineheight[1]{\fontsize{\fsize}{#1\fsize}\selectfont}%
  \ifx\svgwidth\undefined%
    \setlength{\unitlength}{115.63580731bp}%
    \ifx\svgscale\undefined%
      \relax%
    \else%
      \setlength{\unitlength}{\unitlength * \real{\svgscale}}%
    \fi%
  \else%
    \setlength{\unitlength}{\svgwidth}%
  \fi%
  \global\let\svgwidth\undefined%
  \global\let\svgscale\undefined%
  \makeatother%
  \begin{picture}(1,0.85840823)%
    \lineheight{1}%
    \setlength\tabcolsep{0pt}%
    \put(0,0){\includegraphics[width=\unitlength,page=1]{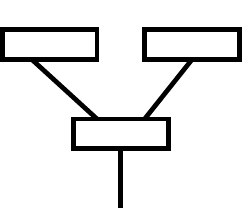}}%
    \put(0.51009212,0.78189609){\color[rgb]{0,0,0}\makebox(0,0)[lt]{\lineheight{1.25}\smash{\begin{tabular}[t]{l}$\scriptscriptstyle 1$\end{tabular}}}}%
    \put(0.38126686,0.2783828){\color[rgb]{0,0,0}\makebox(0,0)[lt]{\lineheight{1.25}\smash{\begin{tabular}[t]{l}$\scriptscriptstyle n-2$\end{tabular}}}}%
    \put(0,0){\includegraphics[width=\unitlength,page=2]{2.1-JWrecursionlemma2-RHS-term2.pdf}}%
    \put(0.72469746,0.39486282){\color[rgb]{0,0,0}\rotatebox{52.127712}{\makebox(0,0)[lt]{\lineheight{1.25}\smash{\begin{tabular}[t]{l}$\scriptscriptstyle m+1$\end{tabular}}}}}%
  \end{picture}%
\endgroup%

}.
\end{equation}
\end{lemma}

\begin{proof}
Apply Lemma \ref{triangle} and then apply the $JW$ recursion (\ref{JW recursion}) once on the resulting $JW_{n-m-1}$ before using the absorption property (\ref{absorption}).
\end{proof}

\section{The Jones-Wenzl projector $JW_{2n-1}$ at a root of unity}

For the rest of the paper, we assume that $q$ is a root of unity and that $n$ is the smallest positive integer such that $q^n \in \{-1,1\}.$ In this case, for $n>1$ we have that $[n]=0$ and $[k] \neq 0$ for all $1 \leq k <n.$

The recursive formula (\ref{JW recursion}) for Jones-Wenzl projectors guarantees that the projectors $JW_1,\dots, JW_{n-1}$ exist. Although for $n>1$ we have $[n]=0,$ the element $JW_{2n-1}$ exists and is given by the following formula in terms of $JW_{n-1}$.

\begin{equation}\label{JW_2n-1}
\begin{tikzpicture}[baseline=6ex]
\draw [ultra thick] (0,0)--(0,2);
\node [draw, minimum width=.75cm, ultra thick, fill = white] (node) at (0,1) {$2n-1$};
\end{tikzpicture}=
\begin{tikzpicture}[baseline=6ex]
\draw [ultra thick] (0,0)--(0,2);
\node [draw, minimum width=.75cm, ultra thick, fill = white] (node) at (0,1) {$n-1$};
\draw [ultra thick] (.75,0)--(.75,2);
\draw [ultra thick] (1.5,0)--(1.5,2);
\node [draw, minimum width=.75cm, ultra thick, fill = white] (node) at (1.5,1) {$n-1$};
\end{tikzpicture}+\sum_{k=1}^{n-1}(-1)^k\left(
\begin{tikzpicture}[baseline=6ex]
\draw [ultra thick] (0,0)--(0,2);
\draw [ultra thick] (.75,1.4) arc (0: -180: .25);
\draw [ultra thick] (.75, 1.4)--(.75,2);
\draw [ultra thick] (1.5,0)--(1.5,2);
\draw [ultra thick] (.75,.6) arc (180:0:.25);
\draw [ultra thick] (.75,0)--(.75,.6);
\draw [ultra thick] (.4,0)--(1.1,2);
\node (node) at (.75,2) [above] {$k$};
\node (node) at (1.1,2) [above] {$1$};
\node (node) at (.75,0) [below] {$k$};
\node (node) at (.4,0) [below] {$1$};
\node [draw, minimum width=.75cm, ultra thick, fill=white] (node) at (0,1.5) {$n-1$};
\node [draw, minimum width=.75cm, ultra thick, fill = white] (node) at (1.5,0.5) {$n-1$};
\end{tikzpicture}+
\begin{tikzpicture}[baseline=6ex]
\draw [ultra thick] (0,0)--(0,2);
\draw [ultra thick] (1.5,0)--(1.5,2);
\draw [ultra thick] (.75,.6) arc (0:180:.25);
\draw [ultra thick] (.75,.6)--(.75,0);
\draw [ultra thick] (.75,1.4) arc (-180:0:.25);
\draw [ultra thick] (.75,1.4)--(.75,2);
\draw [ultra thick] (1.1,0)--(.4,2);
\node (node) at (.4,2) [above] {$1$};
\node (node) at (.75,2) [above] {$k$};
\node (node) at (.75,0) [below] {$k$};
\node (node) at (1.1,0) [below] {$1$};
\node [draw, minimum width=.75cm, ultra thick, fill=white] (node) at (0,0.5) {$n-1$};
\node [draw, minimum width=.75cm, ultra thick, fill = white] (node) at (1.5,1.5) {$n-1$};
\end{tikzpicture}\right).
\end{equation}

Later, in the proof of Theorem \ref{Steinberg skein identities} we will make use of the following notation for the diagrams appearing in (\ref{JW_2n-1}):
\begin{equation}\label{JW diagram notation}
JW_{2n-1}=A_0+\sum_{k=1}^{n-1} (-1)^k(A_k+A_k').    
\end{equation}

It is easy to check that the element defined by (\ref{JW_2n-1}) satisfies the axioms \ref{axiom 1} and \ref{axiom 2} for $JW_{2n-1}$ using the fact that $JW_{n-1}$ does. See \cite[Proposition 2.7]{MS22} for a proof. The formula (\ref{JW_2n-1}) was found already (and in a more general setting of positive characteristic) from representation theoretic techniques in \cite{MS22}. In our setting, it is also possible to obtain the formula (\ref{JW_2n-1}) by specializing the expression for generic $q$ of $JW_{2n-1}$ from \cite{Mor17} in which coefficients of diagrams are certain ratios of quantum binomial coefficients. After canceling denominators of $[n]$ the coefficients can be specialized for $q$ a root of unity. Thus, we view the existence of $JW_{2n-1}$ at a root of unity as encoding ``most" of the miraculous cancellations from \cite{BW16,Bon19}, with the later Theorem \ref{Steinberg skein identities} and Lemma \ref{lem:ncross} encoding the rest of the cancellations.

\section{Quantum $\mathrm{SL}_2$ at a root of unity}\label{Quantum group root of unity}

We recall the fact that a quantum group admits a Frobenius homomorphism at a root of unity \cite{Lus90}. Our results will not mathematically rely on this fact, but a consequence of it serves as inspiration for our main skein identities.

Suppose $q$ is a root of unity and $n$ is the smallest positive integer such that $q^n \in \{-1,1\}.$ Set $t=q^{n^2} \in \{-1,1\}.$ For later use, we note here that it follows from checking cases that

\begin{equation}\label{t,q relation}
(-1)^{n-1}q^n=t.
\end{equation}

Lusztig's version \cite{Lus90} of the quantum group $U_q(sl_2)$ admits a Hopf algebra map called the Frobenius map denoted by $Fr: U_q(sl_2) \rightarrow U_t(sl_2).$ Similarly, Parshall and Wang \cite{PW91} describe a Hopf algebra map for the dual situation $Fr: \mathcal{O}_t(\mathrm{SL}_2) \rightarrow \mathcal{O}_q(\mathrm{SL}_2).$ The existence of either map gives a ribbon functor between the categories of representations $Fr: \text{Rep}_t(\mathrm{SL}_2) \rightarrow \text{Rep}_q(\mathrm{SL}_2).$

The representation category $\text{Rep}_t(\mathrm{SL}_2)$ is semisimple while $\text{Rep}_q(\mathrm{SL}_2)$ is not, unless $q=\pm 1.$ Let $V^{(t)}=V_1^{(t)}$ denote the defining $2\text{-dimensional}$ representation of $U_t(sl_2)$ and $V_k^{(t)}=\text{Sym}^k(V^{(t)})$ denote the $k+1\text{-dimensional}$ irreducible representation of $U_t(sl_2).$ Similarly, let $V^{(q)}$ be the defining representation of $U_q(sl_2)$ and let $V_m^{(q)}$ denote the tilting module $\text{Sym}^m(V)$ whenever $m\equiv -1 \mod{n}$. The representation $V_{n-1}^{(q)}$ is a tilting module known as the Steinberg module and corresponds to $JW_{n-1}$. Similarly, the modules $V_{n-1+kn}^{(q)}$ are tilting and correspond to $JW_{n-1+kn}.$ The following identities hold for all $k\geq 1$ and are known as Steinberg tensor product formulas \cite{AW92}:

\begin{equation}\label{Tensor product formulae}
V_{n-1}^{(q)} \otimes Fr(V_k^{(t)})\cong V_{n-1+kn}^{(q)}.
\end{equation}

For more discussion of these formulas in the specific cases of $\mathrm{SL}_2$, see \cite[Section 2.5]{Ost08}.

\section{First Steinberg skein identities}

In the theory of skein modules, strands which appear to be labeled by the Frobenius representation arise in two common ways. For a knot $K$, the Frobenius image is $K^{[T_n]}$, the knot threaded by the polynomial $T_n.$ For an arc in the context of Muller skein algebras or stated skein algebras \cite{Mul16,LP19,BL22,KQ19}, the Frobenius image is $n$ parallel copies of the arc. In this section we will prove the following skein identities.

\begin{theorem}\label{Steinberg skein identities}
If $q$ is a root of unity and $n$ is the smallest positive integer such that $q^{n} \in \{-1,1\}$ then the following local skein identities hold in $\Sq(M)$, the skein module of any oriented $3\text{-manifold}$ $M.$

\begin{equation*}
\begin{tikzpicture}[baseline=6ex]
\draw [ultra thick] (0,0)--(0,2);
\node [draw, minimum width=.75cm, ultra thick, fill = white] (node) at (0,1) {$n-1$};
\draw [ultra thick] (.9,.25)--(.9,1.75);
\draw [ultra thick] (.9,1.75) arc (180:0:.25);
\draw [ultra thick] (1.4,1.75)--(1.4,.25);
\draw [ultra thick] (.9,.25) arc (-180:0:.25);
\filldraw [fill=white, draw=black] (1.15,1) circle [radius=.075];
\node (node) at (.9,1.75) [left] {$T_n$};
\end{tikzpicture}=
\begin{tikzpicture}[baseline=6ex]
\draw [ultra thick] (0,0)--(0,2);
\draw [ultra thick] (1.4,1.75) arc (0:90:.25);
\draw [ultra thick] (1.15,2) -- (.65,2);
\draw [ultra thick] (.65,2) arc (90:180:.25);
\draw [ultra thick] (.4,1.75)--(.4,.25);
\draw [ultra thick] (1.4,1.75)--(1.4,.25);
\draw [ultra thick] (1.4,.25) arc (0:-90:.25);
\draw [ultra thick] (1.15,0)--(.65,0);
\draw [ultra thick] (.65,0) arc (-90:-180:.25);
\node [draw, minimum width=.75cm, ultra thick, fill = white] (node) at (0,1) {$2n-1$};
\filldraw [fill=white, draw=black] (1.15,1) circle [radius=.075];
\node (node) at (.9,1.75) [right] {$n$};
\end{tikzpicture}
\text{ and }
\begin{tikzpicture}[baseline=6ex]
\draw [ultra thick] (0,0)--(0,2);
\node [draw, minimum width=.75cm, ultra thick, fill = white] (node) at (0,1) {$n-1$};
\draw [ultra thick] (.9,.25)--(.9,1.75);
\node (node) at (.9,1) [right] {$n$};
\node [draw, minimum width=.5cm, ultra thick, fill = white] (node) at (.9,1.75) {};
\node [draw, minimum width=.5cm, ultra thick, fill = white] (node) at (.9,.25) {};
\end{tikzpicture}=
\begin{tikzpicture}[baseline=6ex]
\draw [ultra thick] (-.25,0)--(-.25,2);
\draw [ultra thick] (.9, 1.75)--(.25,1);
\draw [ultra thick] (.9,.25)--(.25,1);
\node [draw, minimum width=.75cm, ultra thick, fill = white] (node) at (0,1) {$2n-1$};
\node [draw, minimum width=.5cm, ultra thick, fill = white] (node) at (.9,1.75) {};
\node [draw, minimum width=.5cm, ultra thick, fill = white] (node) at (.9,.25) {};
\node (node) at (.55,1.4) [right] {$n$};
\node (node) at (.55,.6) [right] {$n$};
\end{tikzpicture}.
\end{equation*}
\end{theorem}

Since the skein module is functorial with respect to embeddings \cite{Prz91}, the first identity can be applied to any framed knot in $M.$

\begin{notation}
We introduce the following notation involving green strands. A green knot is equal to the knot threaded by the polynomial $T_n$ and a green arc ending at $JW$ projectors is equal to $n$ parallel copies of the arc.

\begin{equation}\label{green strands}
\begin{tikzpicture}[baseline=6ex]
\draw [ultra thick, color=ForestGreen] (1.4,1) arc (0:360:.4);
\filldraw [fill=white, draw=black] (1,1) circle [radius=.075];
\end{tikzpicture}=
\begin{tikzpicture}[baseline=6ex]
\draw [ultra thick] (1.4,1) arc (0:360:.4);
\filldraw [fill=white, draw=black] (1,1) circle [radius=.075];
\node (node) at (1,1.4) [above] {$T_n$};
\end{tikzpicture} \text{ and }
\begin{tikzpicture}[baseline=6ex]
\draw [ultra thick, color=ForestGreen] (.9,.25)--(.9,1.75);
\node [draw, minimum width=.5cm, ultra thick, fill = white] (node) at (.9,1.75) {};
\node [draw, minimum width=.5cm, ultra thick, fill = white] (node) at (.9,.25) {};    
\end{tikzpicture}=
\begin{tikzpicture}[baseline=6ex]
\node (node) at (.9,1) [right] {$n$};
\draw [ultra thick] (.9,.25)--(.9,1.75);
\node [draw, minimum width=.5cm, ultra thick, fill = white] (node) at (.9,1.75) {};
\node [draw, minimum width=.5cm, ultra thick, fill = white] (node) at (.9,.25) {};    
\end{tikzpicture}.
\end{equation}
\end{notation}

Using green strands, both identities of Theorem \ref{Steinberg skein identities} can be locally depicted by the single identity

\begin{equation}\label{(green strand Steinberg identity)}
\begin{tikzpicture}[baseline=6ex]
\draw [ultra thick] (0,0)--(0,2);
\node [draw, minimum width=.75cm, ultra thick, fill = white] (node) at (0,1) {$n-1$};
\draw [ultra thick, color=ForestGreen] (.9,0)--(.9,2);
\end{tikzpicture}=
\begin{tikzpicture}[baseline=6ex]
\draw [ultra thick] (-.25,0)--(-.25,2);
\draw [ultra thick, color=ForestGreen] (.9, 2)--(.25,1);
\draw [ultra thick, color=ForestGreen] (.9,0)--(.25,1);
\node [draw, minimum width=.75cm, ultra thick, fill = white] (node) at (0,1) {$2n-1$};
\node (node) at (-.25,2) [above] {$n-1$};
\node (node) at (-.25,0) [below] {$n-1$};
\end{tikzpicture}.
\end{equation}

We view these skein identities as incarnations of the first Steinberg tensor product formula from Equation (\ref{Tensor product formulae}) which states

\begin{equation}
V_{n-1}^{(q)} \otimes Fr(V^{(t)})\cong V_{2n-1}^{(q)}.
\end{equation}

Here $V_{n-1}^{(q)}$ corresponds diagrammatically to $JW_{n-1}$ and $V_{2n-1}^{(q)}$  corresponds to the element $JW_{2n-1}$ given in (\ref{JW_2n-1}).

\begin{proof}[Proof of Theorem \ref{Steinberg skein identities}]
The second identity follows from applying the definition of $JW_{2n-1}$ (\ref{JW_2n-1}) to the right side of the identity. After applying the absorption property (\ref{absorption}) of $JW$ projectors and killing diagrams with a cap attached to a $JW$ projector, the only surviving term is the left hand side of the second identity.

For the first identity, we begin with the right hand side. Recall the notation of $A_j$ from (\ref{JW diagram notation}), in which we write $JW_{2n-1}=A_0+\sum_{j=1}^{n-1}(-1)^j(A_j+A_j')$. We then compute

\begin{align*}
&\sum_{j=1}^{n-1} (-1)^j \centerv { \def\svgscale{0.6}
\begingroup%
  \makeatletter%
  \providecommand\color[2][]{%
    \errmessage{(Inkscape) Color is used for the text in Inkscape, but the package 'color.sty' is not loaded}%
    \renewcommand\color[2][]{}%
  }%
  \providecommand\transparent[1]{%
    \errmessage{(Inkscape) Transparency is used (non-zero) for the text in Inkscape, but the package 'transparent.sty' is not loaded}%
    \renewcommand\transparent[1]{}%
  }%
  \providecommand\rotatebox[2]{#2}%
  \newcommand*\fsize{\dimexpr\f@size pt\relax}%
  \newcommand*\lineheight[1]{\fontsize{\fsize}{#1\fsize}\selectfont}%
  \ifx\svgwidth\undefined%
    \setlength{\unitlength}{129.3495022bp}%
    \ifx\svgscale\undefined%
      \relax%
    \else%
      \setlength{\unitlength}{\unitlength * \real{\svgscale}}%
    \fi%
  \else%
    \setlength{\unitlength}{\svgwidth}%
  \fi%
  \global\let\svgwidth\undefined%
  \global\let\svgscale\undefined%
  \makeatother%
  \begin{picture}(1,0.87658495)%
    \lineheight{1}%
    \setlength\tabcolsep{0pt}%
    \put(0,0){\includegraphics[width=\unitlength,page=1]{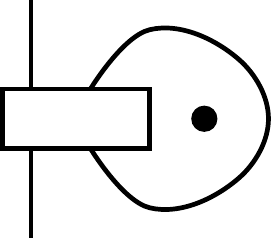}}%
    \put(0.07088844,0.59169482){\color[rgb]{0,0,0}\rotatebox{90}{\makebox(0,0)[lt]{\lineheight{1.25}\smash{\begin{tabular}[t]{l}$\scriptstyle n-1$\end{tabular}}}}}%
    \put(0.0734854,0.04029221){\color[rgb]{0,0,0}\rotatebox{90}{\makebox(0,0)[lt]{\lineheight{1.25}\smash{\begin{tabular}[t]{l}$\scriptstyle n-1$\end{tabular}}}}}%
    \put(0.08926594,0.41378095){\color[rgb]{0,0,0}\makebox(0,0)[lt]{\lineheight{1.25}\smash{\begin{tabular}[t]{l}$A_j$\end{tabular}}}}%
    \put(0.61521689,0.67675644){\color[rgb]{0,0,0}\makebox(0,0)[lt]{\lineheight{1.25}\smash{\begin{tabular}[t]{l}$\scriptstyle n$\end{tabular}}}}%
  \end{picture}%
\endgroup%

}= \sum_{j=1}^{n-1} (-1)^j\centerv { \def\svgscale{0.5}
\begingroup%
  \makeatletter%
  \providecommand\color[2][]{%
    \errmessage{(Inkscape) Color is used for the text in Inkscape, but the package 'color.sty' is not loaded}%
    \renewcommand\color[2][]{}%
  }%
  \providecommand\transparent[1]{%
    \errmessage{(Inkscape) Transparency is used (non-zero) for the text in Inkscape, but the package 'transparent.sty' is not loaded}%
    \renewcommand\transparent[1]{}%
  }%
  \providecommand\rotatebox[2]{#2}%
  \newcommand*\fsize{\dimexpr\f@size pt\relax}%
  \newcommand*\lineheight[1]{\fontsize{\fsize}{#1\fsize}\selectfont}%
  \ifx\svgwidth\undefined%
    \setlength{\unitlength}{267.13360644bp}%
    \ifx\svgscale\undefined%
      \relax%
    \else%
      \setlength{\unitlength}{\unitlength * \real{\svgscale}}%
    \fi%
  \else%
    \setlength{\unitlength}{\svgwidth}%
  \fi%
  \global\let\svgwidth\undefined%
  \global\let\svgscale\undefined%
  \makeatother%
  \begin{picture}(1,0.58362377)%
    \lineheight{1}%
    \setlength\tabcolsep{0pt}%
    \put(0,0){\includegraphics[width=\unitlength,page=1]{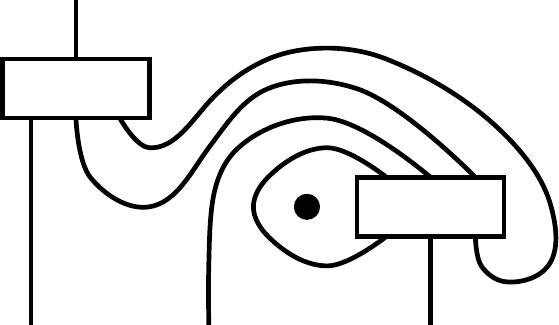}}%
    \put(0.04322379,0.41258483){\color[rgb]{0,0,0}\makebox(0,0)[lt]{\lineheight{1.25}\smash{\begin{tabular}[t]{l}$\scriptstyle n-1$\end{tabular}}}}%
    \put(0.67100555,0.2016155){\color[rgb]{0,0,0}\makebox(0,0)[lt]{\lineheight{1.25}\smash{\begin{tabular}[t]{l}$\scriptstyle n-1$\end{tabular}}}}%
    \put(0.46939015,0.21222687){\color[rgb]{0,0,0}\rotatebox{27.301137}{\makebox(0,0)[lt]{\lineheight{1.25}\smash{\begin{tabular}[t]{l}$\scriptscriptstyle n-j-1$\end{tabular}}}}}%
    \put(0.68896999,0.49747555){\color[rgb]{0,0,0}\makebox(0,0)[lt]{\lineheight{1.25}\smash{\begin{tabular}[t]{l}$\scriptstyle 1$\end{tabular}}}}%
    \put(0.15104986,0.16978148){\color[rgb]{0,0,0}\makebox(0,0)[lt]{\lineheight{1.25}\smash{\begin{tabular}[t]{l}$\scriptscriptstyle j-1$\end{tabular}}}}%
    \put(0.32972999,0.03057661){\color[rgb]{0,0,0}\makebox(0,0)[lt]{\lineheight{1.25}\smash{\begin{tabular}[t]{l}$\scriptstyle 1$\end{tabular}}}}%
    \put(0.751685,0.0088987){\color[rgb]{0,0,0}\rotatebox{90}{\makebox(0,0)[lt]{\lineheight{1.25}\smash{\begin{tabular}[t]{l}$\scriptscriptstyle j-1$\end{tabular}}}}}%
    \put(0.03558264,0.01951004){\color[rgb]{0,0,0}\rotatebox{90}{\makebox(0,0)[lt]{\lineheight{1.25}\smash{\begin{tabular}[t]{l}$\scriptscriptstyle n-j-1$\end{tabular}}}}}%
  \end{picture}%
\endgroup%

}\\
&=\sum_{j=1}^{n-1} (-1)^j\left(\tfrac{[n-j]}{[n-1]}
\centerv { \def\svgscale{0.5}
\begingroup%
  \makeatletter%
  \providecommand\color[2][]{%
    \errmessage{(Inkscape) Color is used for the text in Inkscape, but the package 'color.sty' is not loaded}%
    \renewcommand\color[2][]{}%
  }%
  \providecommand\transparent[1]{%
    \errmessage{(Inkscape) Transparency is used (non-zero) for the text in Inkscape, but the package 'transparent.sty' is not loaded}%
    \renewcommand\transparent[1]{}%
  }%
  \providecommand\rotatebox[2]{#2}%
  \newcommand*\fsize{\dimexpr\f@size pt\relax}%
  \newcommand*\lineheight[1]{\fontsize{\fsize}{#1\fsize}\selectfont}%
  \ifx\svgwidth\undefined%
    \setlength{\unitlength}{200.6751882bp}%
    \ifx\svgscale\undefined%
      \relax%
    \else%
      \setlength{\unitlength}{\unitlength * \real{\svgscale}}%
    \fi%
  \else%
    \setlength{\unitlength}{\svgwidth}%
  \fi%
  \global\let\svgwidth\undefined%
  \global\let\svgscale\undefined%
  \makeatother%
  \begin{picture}(1,0.70627701)%
    \lineheight{1}%
    \setlength\tabcolsep{0pt}%
    \put(0,0){\includegraphics[width=\unitlength,page=1]{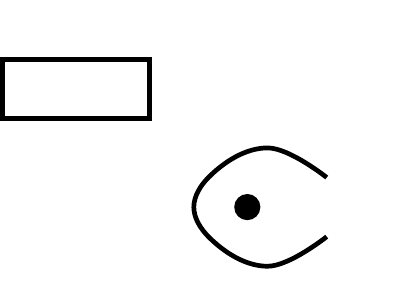}}%
    \put(0.05753839,0.47859443){\color[rgb]{0,0,0}\makebox(0,0)[lt]{\lineheight{1.25}\smash{\begin{tabular}[t]{l}$\scriptstyle n-1$\end{tabular}}}}%
    \put(0.75196978,0.19775753){\color[rgb]{0,0,0}\makebox(0,0)[lt]{\lineheight{1.25}\smash{\begin{tabular}[t]{l}$\scriptstyle n-2$\end{tabular}}}}%
    \put(0.48358455,0.21188315){\color[rgb]{0,0,0}\rotatebox{27.301137}{\makebox(0,0)[lt]{\lineheight{1.25}\smash{\begin{tabular}[t]{l}$\scriptscriptstyle n-j-1$\end{tabular}}}}}%
    \put(0,0){\includegraphics[width=\unitlength,page=2]{2.3-steinberg-prelim-Ajloop-Bj.pdf}}%
    \put(0.57062773,0.49439389){\color[rgb]{0,0,0}\makebox(0,0)[lt]{\lineheight{1.25}\smash{\begin{tabular}[t]{l}$\scriptscriptstyle j-1$\end{tabular}}}}%
    \put(0.89551518,-0.00000005){\color[rgb]{0,0,0}\rotatebox{90}{\makebox(0,0)[lt]{\lineheight{1.25}\smash{\begin{tabular}[t]{l}$\scriptscriptstyle j-1$\end{tabular}}}}}%
    \put(0.07623378,0.07062765){\color[rgb]{0,0,0}\rotatebox{90}{\makebox(0,0)[lt]{\lineheight{1.25}\smash{\begin{tabular}[t]{l}$\scriptscriptstyle n-j$\end{tabular}}}}}%
  \end{picture}%
\endgroup%

}+
\tfrac{[n-j-1]}{[n-1]}
\centerv { \def\svgscale{0.5}
\begingroup%
  \makeatletter%
  \providecommand\color[2][]{%
    \errmessage{(Inkscape) Color is used for the text in Inkscape, but the package 'color.sty' is not loaded}%
    \renewcommand\color[2][]{}%
  }%
  \providecommand\transparent[1]{%
    \errmessage{(Inkscape) Transparency is used (non-zero) for the text in Inkscape, but the package 'transparent.sty' is not loaded}%
    \renewcommand\transparent[1]{}%
  }%
  \providecommand\rotatebox[2]{#2}%
  \newcommand*\fsize{\dimexpr\f@size pt\relax}%
  \newcommand*\lineheight[1]{\fontsize{\fsize}{#1\fsize}\selectfont}%
  \ifx\svgwidth\undefined%
    \setlength{\unitlength}{200.6751882bp}%
    \ifx\svgscale\undefined%
      \relax%
    \else%
      \setlength{\unitlength}{\unitlength * \real{\svgscale}}%
    \fi%
  \else%
    \setlength{\unitlength}{\svgwidth}%
  \fi%
  \global\let\svgwidth\undefined%
  \global\let\svgscale\undefined%
  \makeatother%
  \begin{picture}(1,0.70627701)%
    \lineheight{1}%
    \setlength\tabcolsep{0pt}%
    \put(0,0){\includegraphics[width=\unitlength,page=1]{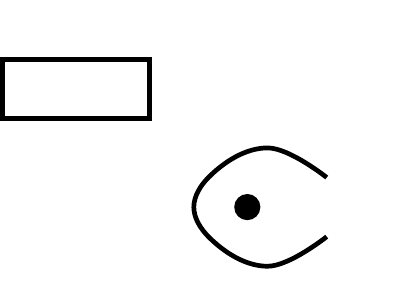}}%
    \put(0.05753839,0.47859445){\color[rgb]{0,0,0}\makebox(0,0)[lt]{\lineheight{1.25}\smash{\begin{tabular}[t]{l}$\scriptstyle n-1$\end{tabular}}}}%
    \put(0.75196972,0.1977575){\color[rgb]{0,0,0}\makebox(0,0)[lt]{\lineheight{1.25}\smash{\begin{tabular}[t]{l}$\scriptstyle n-2$\end{tabular}}}}%
    \put(0.48358455,0.21188316){\color[rgb]{0,0,0}\rotatebox{27.301137}{\makebox(0,0)[lt]{\lineheight{1.25}\smash{\begin{tabular}[t]{l}$\scriptscriptstyle n-j-2$\end{tabular}}}}}%
    \put(0,0){\includegraphics[width=\unitlength,page=2]{2.3-steinberg-prelim-Ajloop-Bj+1.pdf}}%
    \put(0.58475321,0.48026832){\color[rgb]{0,0,0}\makebox(0,0)[lt]{\lineheight{1.25}\smash{\begin{tabular}[t]{l}$\scriptscriptstyle j$\end{tabular}}}}%
    \put(0.89551512,-0.00000008){\color[rgb]{0,0,0}\rotatebox{90}{\makebox(0,0)[lt]{\lineheight{1.25}\smash{\begin{tabular}[t]{l}$\scriptscriptstyle j$\end{tabular}}}}}%
    \put(0.07623378,0.07062762){\color[rgb]{0,0,0}\rotatebox{90}{\makebox(0,0)[lt]{\lineheight{1.25}\smash{\begin{tabular}[t]{l}$\scriptscriptstyle n-j-1$\end{tabular}}}}}%
  \end{picture}%
\endgroup%

}\right)\\
&=(-1)
\centerv { \def\svgscale{0.6}
\begingroup%
  \makeatletter%
  \providecommand\color[2][]{%
    \errmessage{(Inkscape) Color is used for the text in Inkscape, but the package 'color.sty' is not loaded}%
    \renewcommand\color[2][]{}%
  }%
  \providecommand\transparent[1]{%
    \errmessage{(Inkscape) Transparency is used (non-zero) for the text in Inkscape, but the package 'transparent.sty' is not loaded}%
    \renewcommand\transparent[1]{}%
  }%
  \providecommand\rotatebox[2]{#2}%
  \newcommand*\fsize{\dimexpr\f@size pt\relax}%
  \newcommand*\lineheight[1]{\fontsize{\fsize}{#1\fsize}\selectfont}%
  \ifx\svgwidth\undefined%
    \setlength{\unitlength}{171.20373583bp}%
    \ifx\svgscale\undefined%
      \relax%
    \else%
      \setlength{\unitlength}{\unitlength * \real{\svgscale}}%
    \fi%
  \else%
    \setlength{\unitlength}{\svgwidth}%
  \fi%
  \global\let\svgwidth\undefined%
  \global\let\svgscale\undefined%
  \makeatother%
  \begin{picture}(1,0.66228594)%
    \lineheight{1}%
    \setlength\tabcolsep{0pt}%
    \put(0,0){\includegraphics[width=\unitlength,page=1]{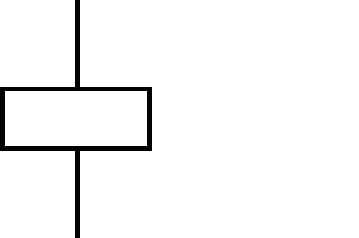}}%
    \put(0.09629334,0.31458582){\color[rgb]{0,0,0}\makebox(0,0)[lt]{\lineheight{1.25}\smash{\begin{tabular}[t]{l}$n-1$\end{tabular}}}}%
    \put(0,0){\includegraphics[width=\unitlength,page=2]{2.3-steinberg-prelim-Ajloop-laststep.pdf}}%
    \put(0.79375713,0.31382505){\color[rgb]{0,0,0}\makebox(0,0)[lt]{\lineheight{1.25}\smash{\begin{tabular}[t]{l}$\scriptstyle n-2$\end{tabular}}}}%
    \put(0,0){\includegraphics[width=\unitlength,page=3]{2.3-steinberg-prelim-Ajloop-laststep.pdf}}%
  \end{picture}%
\endgroup%

}
= (-1)
\centerv { \def\svgscale{0.6}
\begingroup%
  \makeatletter%
  \providecommand\color[2][]{%
    \errmessage{(Inkscape) Color is used for the text in Inkscape, but the package 'color.sty' is not loaded}%
    \renewcommand\color[2][]{}%
  }%
  \providecommand\transparent[1]{%
    \errmessage{(Inkscape) Transparency is used (non-zero) for the text in Inkscape, but the package 'transparent.sty' is not loaded}%
    \renewcommand\transparent[1]{}%
  }%
  \providecommand\rotatebox[2]{#2}%
  \newcommand*\fsize{\dimexpr\f@size pt\relax}%
  \newcommand*\lineheight[1]{\fontsize{\fsize}{#1\fsize}\selectfont}%
  \ifx\svgwidth\undefined%
    \setlength{\unitlength}{144.04484065bp}%
    \ifx\svgscale\undefined%
      \relax%
    \else%
      \setlength{\unitlength}{\unitlength * \real{\svgscale}}%
    \fi%
  \else%
    \setlength{\unitlength}{\svgwidth}%
  \fi%
  \global\let\svgwidth\undefined%
  \global\let\svgscale\undefined%
  \makeatother%
  \begin{picture}(1,0.78715646)%
    \lineheight{1}%
    \setlength\tabcolsep{0pt}%
    \put(0,0){\includegraphics[width=\unitlength,page=1]{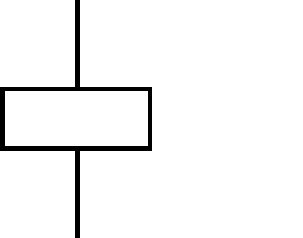}}%
    \put(0.11444895,0.37389932){\color[rgb]{0,0,0}\makebox(0,0)[lt]{\lineheight{1.25}\smash{\begin{tabular}[t]{l}$n-1$\end{tabular}}}}%
    \put(0.72494348,0.60771423){\color[rgb]{0,0,0}\makebox(0,0)[lt]{\lineheight{1.25}\smash{\begin{tabular}[t]{l}$S_{n-2}$\end{tabular}}}}%
    \put(0,0){\includegraphics[width=\unitlength,page=2]{2.3-steinberg-prelim-Ajloop-sum-Sn2.pdf}}%
  \end{picture}%
\endgroup%

}.
\end{align*}

The first equality is given by applying the definition of $A_j$ and applying an isotopy. The second equality is given by applying Lemma \ref{telescoping}. The final line is deduced by recognizing that the alternating sum telescopes. The corresponding computation involving $A_j'$ gives the same result by symmetry. The contribution of the $A_0$ term to the polynomial threading the knot is $xS_{n-1}.$ Thus, we conclude that

\begin{equation*}
\begin{tikzpicture}[baseline=6ex]
\draw [ultra thick] (0,0)--(0,2);
\draw [ultra thick] (1.4,1.75) arc (0:90:.25);
\draw [ultra thick] (1.15,2) -- (.65,2);
\draw [ultra thick] (.65,2) arc (90:180:.25);
\draw [ultra thick] (.4,1.75)--(.4,.25);
\draw [ultra thick] (1.4,1.75)--(1.4,.25);
\draw [ultra thick] (1.4,.25) arc (0:-90:.25);
\draw [ultra thick] (1.15,0)--(.65,0);
\draw [ultra thick] (.65,0) arc (-90:-180:.25);
\node [draw, minimum width=.75cm, ultra thick, fill = white] (node) at (0,1) {$2n-1$};
\filldraw [fill=white, draw=black] (1.15,1) circle [radius=.075];
\node (node) at (.9,1.75) [right] {$n$};
\end{tikzpicture}=
\begin{tikzpicture}[baseline=6ex]
\draw [ultra thick] (0,0)--(0,2);
\node [draw, minimum width=.75cm, ultra thick, fill = white] (node) at (0,1) {$n-1$};
\draw [ultra thick] (.9,.25)--(.9,1.75);
\draw [ultra thick] (.9,1.75) arc (180:0:.25);
\draw [ultra thick] (1.4,1.75)--(1.4,.25);
\draw [ultra thick] (.9,.25) arc (-180:0:.25);
\filldraw [fill=white, draw=black] (1.15,1) circle [radius=.075];
\node (node) at (1.4,1.75) [right] {$xS_{n-1}-2S_{n-2}$};
\end{tikzpicture}.
\end{equation*}

Using the identity $T_n=xS_{n-1}-2S_{n-2}$ from (\ref{Tk and Sk}) we complete the proof.
\end{proof}

\section{A new proof of the Chebyshev-Frobenius homomorphism}

In this section we show that the skein identities from Theorem \ref{Steinberg skein identities} give a new proof of the existence of the Chebyshev-Frobenius homomorphism of Bonahon-Wong \cite{BW16} in Theorem \ref{Chebyshev-Frobenius}. We first begin by showing Theorem \ref{Steinberg skein identities} gives a new proof of the following identity.

\begin{lemma}\label{lem:ncross}
Suppose $q^{1/2}$ is a root of unity and $n$ the smallest positive integer such that $q^{n} \in \{-1,1\}$. Let $t^{1/2}=(q^{1/2})^{n^2}.$ Then the following identity holds.

\begin{equation} \label{eqn:ncross}
\centerv { \def\svgscale{0.5}
\begingroup%
  \makeatletter%
  \providecommand\color[2][]{%
    \errmessage{(Inkscape) Color is used for the text in Inkscape, but the package 'color.sty' is not loaded}%
    \renewcommand\color[2][]{}%
  }%
  \providecommand\transparent[1]{%
    \errmessage{(Inkscape) Transparency is used (non-zero) for the text in Inkscape, but the package 'transparent.sty' is not loaded}%
    \renewcommand\transparent[1]{}%
  }%
  \providecommand\rotatebox[2]{#2}%
  \newcommand*\fsize{\dimexpr\f@size pt\relax}%
  \newcommand*\lineheight[1]{\fontsize{\fsize}{#1\fsize}\selectfont}%
  \ifx\svgwidth\undefined%
    \setlength{\unitlength}{158.15550316bp}%
    \ifx\svgscale\undefined%
      \relax%
    \else%
      \setlength{\unitlength}{\unitlength * \real{\svgscale}}%
    \fi%
  \else%
    \setlength{\unitlength}{\svgwidth}%
  \fi%
  \global\let\svgwidth\undefined%
  \global\let\svgscale\undefined%
  \makeatother%
  \begin{picture}(1,0.64153696)%
    \lineheight{1}%
    \setlength\tabcolsep{0pt}%
    \put(0,0){\includegraphics[width=\unitlength,page=1]{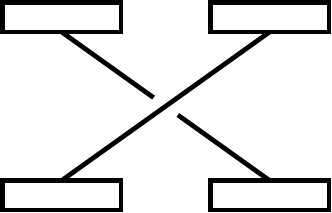}}%
    \put(0.26546933,0.22219113){\color[rgb]{0,0,0}\makebox(0,0)[lt]{\lineheight{1.25}\smash{\begin{tabular}[t]{l}$\scriptstyle n$\end{tabular}}}}%
    \put(0.22219105,0.41258952){\color[rgb]{0,0,0}\makebox(0,0)[lt]{\lineheight{1.25}\smash{\begin{tabular}[t]{l}$\scriptstyle n$\end{tabular}}}}%
  \end{picture}%
\endgroup%

}
= t^{1/2}
\centerv { \def\svgscale{0.5}
\begingroup%
  \makeatletter%
  \providecommand\color[2][]{%
    \errmessage{(Inkscape) Color is used for the text in Inkscape, but the package 'color.sty' is not loaded}%
    \renewcommand\color[2][]{}%
  }%
  \providecommand\transparent[1]{%
    \errmessage{(Inkscape) Transparency is used (non-zero) for the text in Inkscape, but the package 'transparent.sty' is not loaded}%
    \renewcommand\transparent[1]{}%
  }%
  \providecommand\rotatebox[2]{#2}%
  \newcommand*\fsize{\dimexpr\f@size pt\relax}%
  \newcommand*\lineheight[1]{\fontsize{\fsize}{#1\fsize}\selectfont}%
  \ifx\svgwidth\undefined%
    \setlength{\unitlength}{158.15549235bp}%
    \ifx\svgscale\undefined%
      \relax%
    \else%
      \setlength{\unitlength}{\unitlength * \real{\svgscale}}%
    \fi%
  \else%
    \setlength{\unitlength}{\svgwidth}%
  \fi%
  \global\let\svgwidth\undefined%
  \global\let\svgscale\undefined%
  \makeatother%
  \begin{picture}(1,0.641537)%
    \lineheight{1}%
    \setlength\tabcolsep{0pt}%
    \put(0,0){\includegraphics[width=\unitlength,page=1]{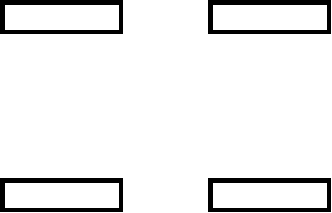}}%
    \put(0.20426794,0.32973009){\color[rgb]{0,0,0}\makebox(0,0)[lt]{\lineheight{1.25}\smash{\begin{tabular}[t]{l}$\scriptstyle n$\end{tabular}}}}%
    \put(0.74436945,0.32973009){\color[rgb]{0,0,0}\makebox(0,0)[lt]{\lineheight{1.25}\smash{\begin{tabular}[t]{l}$\scriptstyle n$\end{tabular}}}}%
    \put(0,0){\includegraphics[width=\unitlength,page=2]{2.4-ncross-end-vertterm.pdf}}%
  \end{picture}%
\endgroup%

}
+ t^{-1/2}
\centerv { \def\svgscale{0.5}
\begingroup%
  \makeatletter%
  \providecommand\color[2][]{%
    \errmessage{(Inkscape) Color is used for the text in Inkscape, but the package 'color.sty' is not loaded}%
    \renewcommand\color[2][]{}%
  }%
  \providecommand\transparent[1]{%
    \errmessage{(Inkscape) Transparency is used (non-zero) for the text in Inkscape, but the package 'transparent.sty' is not loaded}%
    \renewcommand\transparent[1]{}%
  }%
  \providecommand\rotatebox[2]{#2}%
  \newcommand*\fsize{\dimexpr\f@size pt\relax}%
  \newcommand*\lineheight[1]{\fontsize{\fsize}{#1\fsize}\selectfont}%
  \ifx\svgwidth\undefined%
    \setlength{\unitlength}{158.15551397bp}%
    \ifx\svgscale\undefined%
      \relax%
    \else%
      \setlength{\unitlength}{\unitlength * \real{\svgscale}}%
    \fi%
  \else%
    \setlength{\unitlength}{\svgwidth}%
  \fi%
  \global\let\svgwidth\undefined%
  \global\let\svgscale\undefined%
  \makeatother%
  \begin{picture}(1,0.64153692)%
    \lineheight{1}%
    \setlength\tabcolsep{0pt}%
    \put(0,0){\includegraphics[width=\unitlength,page=1]{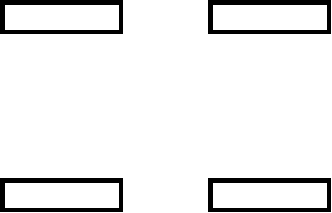}}%
    \put(0.52399186,0.22006716){\color[rgb]{0,0,0}\makebox(0,0)[lt]{\lineheight{1.25}\smash{\begin{tabular}[t]{l}$\scriptstyle n$\end{tabular}}}}%
    \put(0.41934587,0.38349957){\color[rgb]{0,0,0}\makebox(0,0)[lt]{\lineheight{1.25}\smash{\begin{tabular}[t]{l}$\scriptstyle n$\end{tabular}}}}%
    \put(0,0){\includegraphics[width=\unitlength,page=2]{2.4-ncross-end-cupcapterm.pdf}}%
  \end{picture}%
\endgroup%

}
.
\end{equation}

\end{lemma}

\begin{proof}

Starting from the left-hand-side of (\ref{eqn:ncross}), we bring a copy of $JW_{n-1}$ out of one of the blank Jones-Wenzl projectors and then use Theorem \ref{Steinberg skein identities} in the following way.

\begin{equation*}
\centerv { \def\svgscale{0.5}

}
\overset{(\ref{absorption})}{=}
\centerv { \def\svgscale{0.5}
\begingroup%
  \makeatletter%
  \providecommand\color[2][]{%
    \errmessage{(Inkscape) Color is used for the text in Inkscape, but the package 'color.sty' is not loaded}%
    \renewcommand\color[2][]{}%
  }%
  \providecommand\transparent[1]{%
    \errmessage{(Inkscape) Transparency is used (non-zero) for the text in Inkscape, but the package 'transparent.sty' is not loaded}%
    \renewcommand\transparent[1]{}%
  }%
  \providecommand\rotatebox[2]{#2}%
  \newcommand*\fsize{\dimexpr\f@size pt\relax}%
  \newcommand*\lineheight[1]{\fontsize{\fsize}{#1\fsize}\selectfont}%
  \ifx\svgwidth\undefined%
    \setlength{\unitlength}{158.15551397bp}%
    \ifx\svgscale\undefined%
      \relax%
    \else%
      \setlength{\unitlength}{\unitlength * \real{\svgscale}}%
    \fi%
  \else%
    \setlength{\unitlength}{\svgwidth}%
  \fi%
  \global\let\svgwidth\undefined%
  \global\let\svgscale\undefined%
  \makeatother%
  \begin{picture}(1,0.82076818)%
    \lineheight{1}%
    \setlength\tabcolsep{0pt}%
    \put(0,0){\includegraphics[width=\unitlength,page=1]{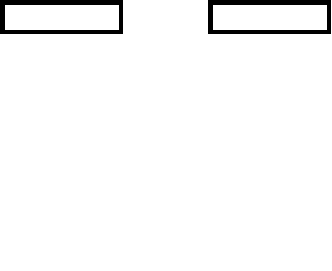}}%
    \put(0.26544976,0.60418217){\color[rgb]{0,0,0}\makebox(0,0)[lt]{\lineheight{1.25}\smash{\begin{tabular}[t]{l}$\scriptstyle n$\end{tabular}}}}%
    \put(0.11175925,0.22006695){\color[rgb]{0,0,0}\makebox(0,0)[lt]{\lineheight{1.25}\smash{\begin{tabular}[t]{l}$\scriptstyle 1$\end{tabular}}}}%
    \put(0,0){\includegraphics[width=\unitlength,page=2]{2.4-ncross-step1.9.pdf}}%
    \put(0.40338436,0.31814423){\color[rgb]{0,0,0}\makebox(0,0)[lt]{\lineheight{1.25}\smash{\begin{tabular}[t]{l}$\scriptscriptstyle n-1$\end{tabular}}}}%
    \put(0,0){\includegraphics[width=\unitlength,page=3]{2.4-ncross-step1.9.pdf}}%
    \put(0.8024178,0.28565195){\color[rgb]{0,0,0}\makebox(0,0)[lt]{\lineheight{1.25}\smash{\begin{tabular}[t]{l}$\scriptstyle n$\end{tabular}}}}%
    \put(0,0){\includegraphics[width=\unitlength,page=4]{2.4-ncross-step1.9.pdf}}%
  \end{picture}%
\endgroup%

}
\overset{(\ref{(green strand Steinberg identity)})}{=}
\centerv { \def\svgscale{0.5}
\begingroup%
  \makeatletter%
  \providecommand\color[2][]{%
    \errmessage{(Inkscape) Color is used for the text in Inkscape, but the package 'color.sty' is not loaded}%
    \renewcommand\color[2][]{}%
  }%
  \providecommand\transparent[1]{%
    \errmessage{(Inkscape) Transparency is used (non-zero) for the text in Inkscape, but the package 'transparent.sty' is not loaded}%
    \renewcommand\transparent[1]{}%
  }%
  \providecommand\rotatebox[2]{#2}%
  \newcommand*\fsize{\dimexpr\f@size pt\relax}%
  \newcommand*\lineheight[1]{\fontsize{\fsize}{#1\fsize}\selectfont}%
  \ifx\svgwidth\undefined%
    \setlength{\unitlength}{158.15550316bp}%
    \ifx\svgscale\undefined%
      \relax%
    \else%
      \setlength{\unitlength}{\unitlength * \real{\svgscale}}%
    \fi%
  \else%
    \setlength{\unitlength}{\svgwidth}%
  \fi%
  \global\let\svgwidth\undefined%
  \global\let\svgscale\undefined%
  \makeatother%
  \begin{picture}(1,0.82076838)%
    \lineheight{1}%
    \setlength\tabcolsep{0pt}%
    \put(0,0){\includegraphics[width=\unitlength,page=1]{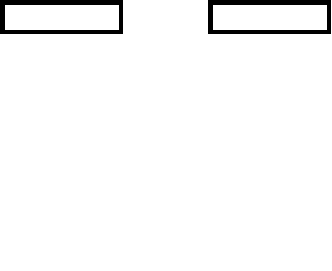}}%
    \put(0.25514449,0.5964533){\color[rgb]{0,0,0}\makebox(0,0)[lt]{\lineheight{1.25}\smash{\begin{tabular}[t]{l}$\scriptstyle n$\end{tabular}}}}%
    \put(0.11175925,0.22006711){\color[rgb]{0,0,0}\makebox(0,0)[lt]{\lineheight{1.25}\smash{\begin{tabular}[t]{l}$\scriptstyle 1$\end{tabular}}}}%
    \put(0.39351964,0.08839563){\color[rgb]{0,0,0}\rotatebox{34.008417}{\makebox(0,0)[lt]{\lineheight{1.25}\smash{\begin{tabular}[t]{l}$\scriptstyle n-1$\end{tabular}}}}}%
    \put(0,0){\includegraphics[width=\unitlength,page=2]{2.4-ncross-step2.pdf}}%
    \put(0.50606874,0.34552919){\color[rgb]{0,0,0}\makebox(0,0)[lt]{\lineheight{1.25}\smash{\begin{tabular}[t]{l}$\scriptstyle 2n-1$\end{tabular}}}}%
    \put(0,0){\includegraphics[width=\unitlength,page=3]{2.4-ncross-step2.pdf}}%
    \put(0.75988571,0.56273088){\color[rgb]{0,0,0}\makebox(0,0)[lt]{\lineheight{1.25}\smash{\begin{tabular}[t]{l}$\scriptstyle n-1$\end{tabular}}}}%
    \put(0.77780887,0.18634468){\color[rgb]{0,0,0}\makebox(0,0)[lt]{\lineheight{1.25}\smash{\begin{tabular}[t]{l}$\scriptstyle n$\end{tabular}}}}%
  \end{picture}%
\endgroup%

}.
\end{equation*}

We then absorb $n(n-1)$ crossings into $JW_{2n-1}$ by using Lemma \ref{crossing absorption} to obtain
\begin{equation*}
q^{n(n-1)/2} 
\centerv { \def\svgscale{0.5}
\begingroup%
  \makeatletter%
  \providecommand\color[2][]{%
    \errmessage{(Inkscape) Color is used for the text in Inkscape, but the package 'color.sty' is not loaded}%
    \renewcommand\color[2][]{}%
  }%
  \providecommand\transparent[1]{%
    \errmessage{(Inkscape) Transparency is used (non-zero) for the text in Inkscape, but the package 'transparent.sty' is not loaded}%
    \renewcommand\transparent[1]{}%
  }%
  \providecommand\rotatebox[2]{#2}%
  \newcommand*\fsize{\dimexpr\f@size pt\relax}%
  \newcommand*\lineheight[1]{\fontsize{\fsize}{#1\fsize}\selectfont}%
  \ifx\svgwidth\undefined%
    \setlength{\unitlength}{158.15550316bp}%
    \ifx\svgscale\undefined%
      \relax%
    \else%
      \setlength{\unitlength}{\unitlength * \real{\svgscale}}%
    \fi%
  \else%
    \setlength{\unitlength}{\svgwidth}%
  \fi%
  \global\let\svgwidth\undefined%
  \global\let\svgscale\undefined%
  \makeatother%
  \begin{picture}(1,0.82076838)%
    \lineheight{1}%
    \setlength\tabcolsep{0pt}%
    \put(0,0){\includegraphics[width=\unitlength,page=1]{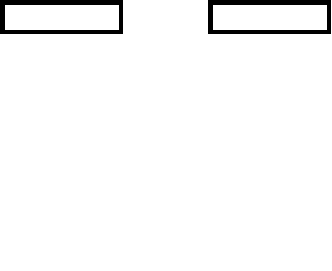}}%
    \put(0.24803131,0.57732591){\color[rgb]{0,0,0}\makebox(0,0)[lt]{\lineheight{1.25}\smash{\begin{tabular}[t]{l}$\scriptstyle n$\end{tabular}}}}%
    \put(0.11175929,0.2200671){\color[rgb]{0,0,0}\makebox(0,0)[lt]{\lineheight{1.25}\smash{\begin{tabular}[t]{l}$\scriptstyle 1$\end{tabular}}}}%
    \put(0.3935197,0.08839556){\color[rgb]{0,0,0}\rotatebox{34.008417}{\makebox(0,0)[lt]{\lineheight{1.25}\smash{\begin{tabular}[t]{l}$\scriptstyle n-1$\end{tabular}}}}}%
    \put(0,0){\includegraphics[width=\unitlength,page=2]{2.4-ncross-step2.1.pdf}}%
    \put(0.50606878,0.34552918){\color[rgb]{0,0,0}\makebox(0,0)[lt]{\lineheight{1.25}\smash{\begin{tabular}[t]{l}$\scriptstyle 2n-1$\end{tabular}}}}%
    \put(0,0){\includegraphics[width=\unitlength,page=3]{2.4-ncross-step2.1.pdf}}%
    \put(0.82446515,0.55192101){\color[rgb]{0,0,0}\makebox(0,0)[lt]{\lineheight{1.25}\smash{\begin{tabular}[t]{l}$\scriptstyle n-1$\end{tabular}}}}%
    \put(0.77780891,0.18634468){\color[rgb]{0,0,0}\makebox(0,0)[lt]{\lineheight{1.25}\smash{\begin{tabular}[t]{l}$\scriptstyle n$\end{tabular}}}}%
    \put(0,0){\includegraphics[width=\unitlength,page=4]{2.4-ncross-step2.1.pdf}}%
  \end{picture}%
\endgroup%

}.
\end{equation*}
We then apply Lemma \ref{JW crossing} to the crossing of the $1$ and $n$ labeled strands to obtain
\begin{equation*}
q^{n(n-1)/2+n/2} 
\centerv { \def\svgscale{0.5}
\begingroup%
  \makeatletter%
  \providecommand\color[2][]{%
    \errmessage{(Inkscape) Color is used for the text in Inkscape, but the package 'color.sty' is not loaded}%
    \renewcommand\color[2][]{}%
  }%
  \providecommand\transparent[1]{%
    \errmessage{(Inkscape) Transparency is used (non-zero) for the text in Inkscape, but the package 'transparent.sty' is not loaded}%
    \renewcommand\transparent[1]{}%
  }%
  \providecommand\rotatebox[2]{#2}%
  \newcommand*\fsize{\dimexpr\f@size pt\relax}%
  \newcommand*\lineheight[1]{\fontsize{\fsize}{#1\fsize}\selectfont}%
  \ifx\svgwidth\undefined%
    \setlength{\unitlength}{158.15551397bp}%
    \ifx\svgscale\undefined%
      \relax%
    \else%
      \setlength{\unitlength}{\unitlength * \real{\svgscale}}%
    \fi%
  \else%
    \setlength{\unitlength}{\svgwidth}%
  \fi%
  \global\let\svgwidth\undefined%
  \global\let\svgscale\undefined%
  \makeatother%
  \begin{picture}(1,0.82076832)%
    \lineheight{1}%
    \setlength\tabcolsep{0pt}%
    \put(0,0){\includegraphics[width=\unitlength,page=1]{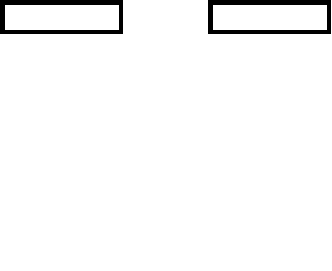}}%
    \put(0.39351967,0.08839555){\color[rgb]{0,0,0}\rotatebox{34.008417}{\makebox(0,0)[lt]{\lineheight{1.25}\smash{\begin{tabular}[t]{l}$\scriptstyle n-1$\end{tabular}}}}}%
    \put(0,0){\includegraphics[width=\unitlength,page=2]{2.4-ncross-step3-term1.pdf}}%
    \put(0.50606871,0.34552917){\color[rgb]{0,0,0}\makebox(0,0)[lt]{\lineheight{1.25}\smash{\begin{tabular}[t]{l}$\scriptstyle 2n-1$\end{tabular}}}}%
    \put(0,0){\includegraphics[width=\unitlength,page=3]{2.4-ncross-step3-term1.pdf}}%
    \put(0.82446512,0.55192098){\color[rgb]{0,0,0}\makebox(0,0)[lt]{\lineheight{1.25}\smash{\begin{tabular}[t]{l}$\scriptstyle n-1$\end{tabular}}}}%
    \put(0.77780888,0.18634467){\color[rgb]{0,0,0}\makebox(0,0)[lt]{\lineheight{1.25}\smash{\begin{tabular}[t]{l}$\scriptstyle n$\end{tabular}}}}%
    \put(0,0){\includegraphics[width=\unitlength,page=4]{2.4-ncross-step3-term1.pdf}}%
    \put(0.6523468,0.61650037){\color[rgb]{0,0,0}\makebox(0,0)[lt]{\lineheight{1.25}\smash{\begin{tabular}[t]{l}$\scriptstyle 1$\end{tabular}}}}%
    \put(0.03223257,0.36557621){\color[rgb]{0,0,0}\makebox(0,0)[lt]{\lineheight{1.25}\smash{\begin{tabular}[t]{l}$\scriptstyle 1$\end{tabular}}}}%
    \put(0.24011427,0.65234671){\color[rgb]{0,0,0}\rotatebox{-42.83218}{\makebox(0,0)[lt]{\lineheight{1.25}\smash{\begin{tabular}[t]{l}$\scriptstyle n-1$\end{tabular}}}}}%
  \end{picture}%
\endgroup%

}
+ q^{n(n-1)/2-n/2} 
\centerv { \def\svgscale{0.5}
\begingroup%
  \makeatletter%
  \providecommand\color[2][]{%
    \errmessage{(Inkscape) Color is used for the text in Inkscape, but the package 'color.sty' is not loaded}%
    \renewcommand\color[2][]{}%
  }%
  \providecommand\transparent[1]{%
    \errmessage{(Inkscape) Transparency is used (non-zero) for the text in Inkscape, but the package 'transparent.sty' is not loaded}%
    \renewcommand\transparent[1]{}%
  }%
  \providecommand\rotatebox[2]{#2}%
  \newcommand*\fsize{\dimexpr\f@size pt\relax}%
  \newcommand*\lineheight[1]{\fontsize{\fsize}{#1\fsize}\selectfont}%
  \ifx\svgwidth\undefined%
    \setlength{\unitlength}{158.15551397bp}%
    \ifx\svgscale\undefined%
      \relax%
    \else%
      \setlength{\unitlength}{\unitlength * \real{\svgscale}}%
    \fi%
  \else%
    \setlength{\unitlength}{\svgwidth}%
  \fi%
  \global\let\svgwidth\undefined%
  \global\let\svgscale\undefined%
  \makeatother%
  \begin{picture}(1,0.82076832)%
    \lineheight{1}%
    \setlength\tabcolsep{0pt}%
    \put(0,0){\includegraphics[width=\unitlength,page=1]{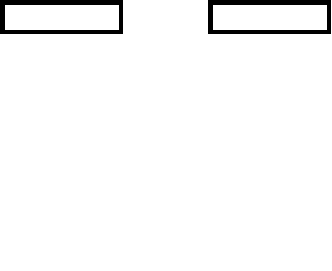}}%
    \put(0.39351968,0.08839554){\color[rgb]{0,0,0}\rotatebox{34.008417}{\makebox(0,0)[lt]{\lineheight{1.25}\smash{\begin{tabular}[t]{l}$\scriptstyle n-1$\end{tabular}}}}}%
    \put(0,0){\includegraphics[width=\unitlength,page=2]{2.4-ncross-step3-term2.pdf}}%
    \put(0.50606871,0.34552917){\color[rgb]{0,0,0}\makebox(0,0)[lt]{\lineheight{1.25}\smash{\begin{tabular}[t]{l}$\scriptstyle 2n-1$\end{tabular}}}}%
    \put(0,0){\includegraphics[width=\unitlength,page=3]{2.4-ncross-step3-term2.pdf}}%
    \put(0.82446512,0.55192098){\color[rgb]{0,0,0}\makebox(0,0)[lt]{\lineheight{1.25}\smash{\begin{tabular}[t]{l}$\scriptstyle n-1$\end{tabular}}}}%
    \put(0.77780888,0.18634467){\color[rgb]{0,0,0}\makebox(0,0)[lt]{\lineheight{1.25}\smash{\begin{tabular}[t]{l}$\scriptstyle n$\end{tabular}}}}%
    \put(0,0){\includegraphics[width=\unitlength,page=4]{2.4-ncross-step3-term2.pdf}}%
    \put(0.4702224,0.66814588){\color[rgb]{0,0,0}\makebox(0,0)[lt]{\lineheight{1.25}\smash{\begin{tabular}[t]{l}$\scriptstyle 1$\end{tabular}}}}%
    \put(0.07591301,0.22006702){\color[rgb]{0,0,0}\makebox(0,0)[lt]{\lineheight{1.25}\smash{\begin{tabular}[t]{l}$\scriptstyle 1$\end{tabular}}}}%
    \put(0.04295959,0.65234669){\color[rgb]{0,0,0}\rotatebox{-34.896941}{\makebox(0,0)[lt]{\lineheight{1.25}\smash{\begin{tabular}[t]{l}$\scriptstyle n-1$\end{tabular}}}}}%
    \put(0,0){\includegraphics[width=\unitlength,page=5]{2.4-ncross-step3-term2.pdf}}%
  \end{picture}%
\endgroup%

}.
\end{equation*}

We then replace $JW_{2n-1}$ in each diagram by the formula (\ref{JW_2n-1}) and use the properties (\ref{absorption}) and \ref{axiom 1}. For the first diagram, only the $A_0$ term survives and for the second diagram, only the $A_{n-1}'$ term survives, leaving us with

\begin{equation*}
q^{n^2/2} 
\centerv { \def\svgscale{0.5}

}
+ (-1)^{n-1}q^{n^2/2-n}
\centerv { \def\svgscale{0.5}

} .
\end{equation*} 

This is the desired result since $t^{1/2}=(q^{1/2})^{n^2}$ and $(-1)^{n-1}q^{-n}=t^{-1}$ by (\ref{t,q relation}).
\end{proof}

\begin{theorem}[Bonahon-Wong \cite{BW16}]\label{Chebyshev-Frobenius}
Let $M$ be an oriented $3\text{-manifold}.$ Suppose $q^{1/2}$ is a root of unity. Let $n$ be the smallest positive integer such that $q^n \in \{-1,1\}$ and set $t^{1/2}=(q^{1/2})^{n^2}$ so that $t=q^{n^2} \in \{-1,1\}.$ Then there exists a homomorphism  of skein modules
\begin{equation*}
Fr: \mathcal{S}_t(M) \rightarrow \Sq(M)    
\end{equation*} defined on a framed link $L$ by $L \mapsto L^{[T_n]}.$
\end{theorem}

We now present the new proof.

\begin{proof}

We have to check that the map $L \mapsto L^{[T_n]}$ respects both the crossing and loop relations (\ref{Kauffman bracket relations}) which are the defining skein relations of $\mathcal{S}_t(M).$ We must show the following identities hold in $\Sq(M).$

\begin{equation} \label{Frobenius crossing}
\centerv { \def\svgscale{0.4}
\begingroup%
  \makeatletter%
  \providecommand\color[2][]{%
    \errmessage{(Inkscape) Color is used for the text in Inkscape, but the package 'color.sty' is not loaded}%
    \renewcommand\color[2][]{}%
  }%
  \providecommand\transparent[1]{%
    \errmessage{(Inkscape) Transparency is used (non-zero) for the text in Inkscape, but the package 'transparent.sty' is not loaded}%
    \renewcommand\transparent[1]{}%
  }%
  \providecommand\rotatebox[2]{#2}%
  \newcommand*\fsize{\dimexpr\f@size pt\relax}%
  \newcommand*\lineheight[1]{\fontsize{\fsize}{#1\fsize}\selectfont}%
  \ifx\svgwidth\undefined%
    \setlength{\unitlength}{115.14277421bp}%
    \ifx\svgscale\undefined%
      \relax%
    \else%
      \setlength{\unitlength}{\unitlength * \real{\svgscale}}%
    \fi%
  \else%
    \setlength{\unitlength}{\svgwidth}%
  \fi%
  \global\let\svgwidth\undefined%
  \global\let\svgscale\undefined%
  \makeatother%
  \begin{picture}(1,1.24313342)%
    \lineheight{1}%
    \setlength\tabcolsep{0pt}%
    \put(0,0){\includegraphics[width=\unitlength,page=1]{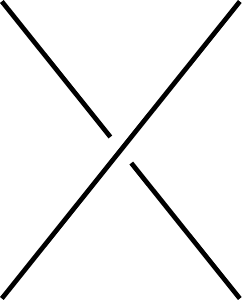}}%
    \put(0.69420953,0.42170102){\color[rgb]{0,0,0}\makebox(0,0)[lt]{\lineheight{1.25}\smash{\begin{tabular}[t]{l}$\scriptstyle T_n$\end{tabular}}}}%
    \put(0.81606742,0.91407159){\color[rgb]{0,0,0}\makebox(0,0)[lt]{\lineheight{1.25}\smash{\begin{tabular}[t]{l}$\scriptstyle T_n$\end{tabular}}}}%
  \end{picture}%
\endgroup%

}
= t^{1/2} 
\centerv { \def\svgscale{0.4}
\begingroup%
  \makeatletter%
  \providecommand\color[2][]{%
    \errmessage{(Inkscape) Color is used for the text in Inkscape, but the package 'color.sty' is not loaded}%
    \renewcommand\color[2][]{}%
  }%
  \providecommand\transparent[1]{%
    \errmessage{(Inkscape) Transparency is used (non-zero) for the text in Inkscape, but the package 'transparent.sty' is not loaded}%
    \renewcommand\transparent[1]{}%
  }%
  \providecommand\rotatebox[2]{#2}%
  \newcommand*\fsize{\dimexpr\f@size pt\relax}%
  \newcommand*\lineheight[1]{\fontsize{\fsize}{#1\fsize}\selectfont}%
  \ifx\svgwidth\undefined%
    \setlength{\unitlength}{115.53914708bp}%
    \ifx\svgscale\undefined%
      \relax%
    \else%
      \setlength{\unitlength}{\unitlength * \real{\svgscale}}%
    \fi%
  \else%
    \setlength{\unitlength}{\svgwidth}%
  \fi%
  \global\let\svgwidth\undefined%
  \global\let\svgscale\undefined%
  \makeatother%
  \begin{picture}(1,1.2323506)%
    \lineheight{1}%
    \setlength\tabcolsep{0pt}%
    \put(0.13198894,0.37083453){\color[rgb]{0,0,0}\makebox(0,0)[lt]{\lineheight{1.25}\smash{\begin{tabular}[t]{l}$\scriptstyle T_n$\end{tabular}}}}%
    \put(0.92138548,0.78500637){\color[rgb]{0,0,0}\makebox(0,0)[lt]{\lineheight{1.25}\smash{\begin{tabular}[t]{l}$\scriptstyle T_n$\end{tabular}}}}%
    \put(0,0){\includegraphics[width=\unitlength,page=1]{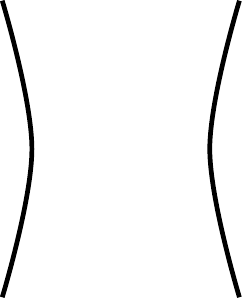}}%
  \end{picture}%
\endgroup%

}
+ t^{-1/2} 
\centerv { \def\svgscale{0.4}
\begingroup%
  \makeatletter%
  \providecommand\color[2][]{%
    \errmessage{(Inkscape) Color is used for the text in Inkscape, but the package 'color.sty' is not loaded}%
    \renewcommand\color[2][]{}%
  }%
  \providecommand\transparent[1]{%
    \errmessage{(Inkscape) Transparency is used (non-zero) for the text in Inkscape, but the package 'transparent.sty' is not loaded}%
    \renewcommand\transparent[1]{}%
  }%
  \providecommand\rotatebox[2]{#2}%
  \newcommand*\fsize{\dimexpr\f@size pt\relax}%
  \newcommand*\lineheight[1]{\fontsize{\fsize}{#1\fsize}\selectfont}%
  \ifx\svgwidth\undefined%
    \setlength{\unitlength}{115.16368716bp}%
    \ifx\svgscale\undefined%
      \relax%
    \else%
      \setlength{\unitlength}{\unitlength * \real{\svgscale}}%
    \fi%
  \else%
    \setlength{\unitlength}{\svgwidth}%
  \fi%
  \global\let\svgwidth\undefined%
  \global\let\svgscale\undefined%
  \makeatother%
  \begin{picture}(1,1.24117465)%
    \lineheight{1}%
    \setlength\tabcolsep{0pt}%
    \put(0.64534167,0.27067049){\color[rgb]{0,0,0}\makebox(0,0)[lt]{\lineheight{1.25}\smash{\begin{tabular}[t]{l}$\scriptstyle T_n$\end{tabular}}}}%
    \put(0.32535892,1.0829343){\color[rgb]{0,0,0}\makebox(0,0)[lt]{\lineheight{1.25}\smash{\begin{tabular}[t]{l}$\scriptstyle T_n$\end{tabular}}}}%
    \put(0,0){\includegraphics[width=\unitlength,page=1]{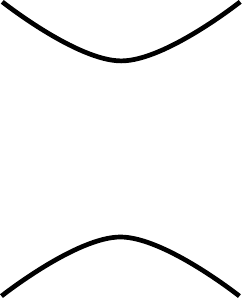}}%
  \end{picture}%
\endgroup%

}
.
\end{equation}

\begin{equation}\label{Frobenius loop}
\begin{tikzpicture}[baseline=0ex]
\draw[ultra thick] (0,0) circle [radius=.75];
\node (node) at (.75,0) [left] {$T_n$};
\end{tikzpicture}=-t-t^{-1}.
\end{equation}

We will use the standard argument that identity (\ref{Frobenius loop}) holds, but will provide a new argument that (\ref{Frobenius crossing}) holds. That (\ref{Frobenius loop}) holds follows by applying equations (\ref{Tk defining property})  and (\ref{t,q relation}) as shown in the following computation:

\begin{equation}
T_n(-q-q^{-1})\overset{(\ref{Tk defining property})}{=}(-1)^nq^n+(-1)^nq^{-n}\overset{(\ref{t,q relation})}{=}-t-t^{-1}.
\end{equation}

To show that identity (\ref{Frobenius crossing}) holds, it will suffice to show that the identity holds in a thickened surface contained in $M$. Furthermore, we may remove a point (or two points in the case of $S^2$) from the surface to obtain a punctured surface $\Sigma$, so that by Lemma \ref{Steinberg loop} and Equation (\ref{JW closed}), the element obtained by closing up $JW_{n-1}$ around a puncture is not a zero divisor in $\Sq(\Sigma).$ Thus, it will suffice to show that the following identity holds in $\Sq(\Sigma).$

\begin{equation} \label{puncture Steinberg}
\centerv { \def\svgscale{0.4}
\begingroup%
  \makeatletter%
  \providecommand\color[2][]{%
    \errmessage{(Inkscape) Color is used for the text in Inkscape, but the package 'color.sty' is not loaded}%
    \renewcommand\color[2][]{}%
  }%
  \providecommand\transparent[1]{%
    \errmessage{(Inkscape) Transparency is used (non-zero) for the text in Inkscape, but the package 'transparent.sty' is not loaded}%
    \renewcommand\transparent[1]{}%
  }%
  \providecommand\rotatebox[2]{#2}%
  \newcommand*\fsize{\dimexpr\f@size pt\relax}%
  \newcommand*\lineheight[1]{\fontsize{\fsize}{#1\fsize}\selectfont}%
  \ifx\svgwidth\undefined%
    \setlength{\unitlength}{174.91685522bp}%
    \ifx\svgscale\undefined%
      \relax%
    \else%
      \setlength{\unitlength}{\unitlength * \real{\svgscale}}%
    \fi%
  \else%
    \setlength{\unitlength}{\svgwidth}%
  \fi%
  \global\let\svgwidth\undefined%
  \global\let\svgscale\undefined%
  \makeatother%
  \begin{picture}(1,0.81831925)%
    \lineheight{1}%
    \setlength\tabcolsep{0pt}%
    \put(0,0){\includegraphics[width=\unitlength,page=1]{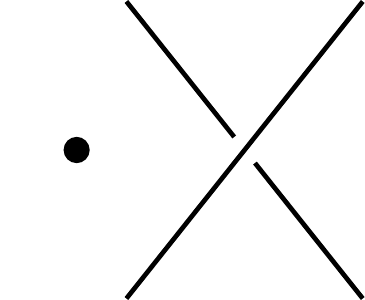}}%
    \put(0.80132244,0.27951426){\color[rgb]{0,0,0}\makebox(0,0)[lt]{\lineheight{1.25}\smash{\begin{tabular}[t]{l}$\scriptstyle T_n$\end{tabular}}}}%
    \put(0.88153807,0.60362769){\color[rgb]{0,0,0}\makebox(0,0)[lt]{\lineheight{1.25}\smash{\begin{tabular}[t]{l}$\scriptstyle T_n$\end{tabular}}}}%
    \put(0,0){\includegraphics[width=\unitlength,page=2]{3.3-Tncross-step3-lhs.pdf}}%
    \put(0.29551815,0.39103349){\color[rgb]{0,0,0}\makebox(0,0)[lt]{\lineheight{1.25}\smash{\begin{tabular}[t]{l}$\scriptstyle n-1$\end{tabular}}}}%
    \put(0,0){\includegraphics[width=\unitlength,page=3]{3.3-Tncross-step3-lhs.pdf}}%
  \end{picture}%
\endgroup%

}
= t^{1/2}
\centerv { \def\svgscale{0.4}
\begingroup%
  \makeatletter%
  \providecommand\color[2][]{%
    \errmessage{(Inkscape) Color is used for the text in Inkscape, but the package 'color.sty' is not loaded}%
    \renewcommand\color[2][]{}%
  }%
  \providecommand\transparent[1]{%
    \errmessage{(Inkscape) Transparency is used (non-zero) for the text in Inkscape, but the package 'transparent.sty' is not loaded}%
    \renewcommand\transparent[1]{}%
  }%
  \providecommand\rotatebox[2]{#2}%
  \newcommand*\fsize{\dimexpr\f@size pt\relax}%
  \newcommand*\lineheight[1]{\fontsize{\fsize}{#1\fsize}\selectfont}%
  \ifx\svgwidth\undefined%
    \setlength{\unitlength}{198.21072556bp}%
    \ifx\svgscale\undefined%
      \relax%
    \else%
      \setlength{\unitlength}{\unitlength * \real{\svgscale}}%
    \fi%
  \else%
    \setlength{\unitlength}{\svgwidth}%
  \fi%
  \global\let\svgwidth\undefined%
  \global\let\svgscale\undefined%
  \makeatother%
  \begin{picture}(1,0.71835026)%
    \lineheight{1}%
    \setlength\tabcolsep{0pt}%
    \put(0.49402712,0.21616336){\color[rgb]{0,0,0}\makebox(0,0)[lt]{\lineheight{1.25}\smash{\begin{tabular}[t]{l}$\scriptstyle T_n$\end{tabular}}}}%
    \put(0.95417479,0.45758855){\color[rgb]{0,0,0}\makebox(0,0)[lt]{\lineheight{1.25}\smash{\begin{tabular}[t]{l}$\scriptstyle T_n$\end{tabular}}}}%
    \put(0,0){\includegraphics[width=\unitlength,page=1]{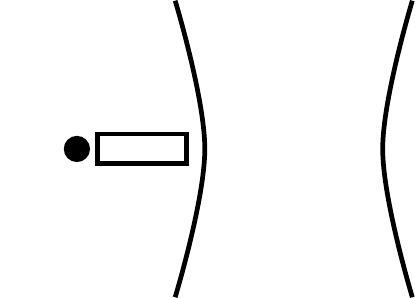}}%
    \put(0.26078868,0.34317917){\color[rgb]{0,0,0}\makebox(0,0)[lt]{\lineheight{1.25}\smash{\begin{tabular}[t]{l}$\scriptstyle n-1$\end{tabular}}}}%
    \put(0,0){\includegraphics[width=\unitlength,page=2]{3.3-Tncross-step3-vertterm.pdf}}%
  \end{picture}%
\endgroup%

}
+ t^{-1/2}
\centerv { \def\svgscale{0.4}
\begingroup%
  \makeatletter%
  \providecommand\color[2][]{%
    \errmessage{(Inkscape) Color is used for the text in Inkscape, but the package 'color.sty' is not loaded}%
    \renewcommand\color[2][]{}%
  }%
  \providecommand\transparent[1]{%
    \errmessage{(Inkscape) Transparency is used (non-zero) for the text in Inkscape, but the package 'transparent.sty' is not loaded}%
    \renewcommand\transparent[1]{}%
  }%
  \providecommand\rotatebox[2]{#2}%
  \newcommand*\fsize{\dimexpr\f@size pt\relax}%
  \newcommand*\lineheight[1]{\fontsize{\fsize}{#1\fsize}\selectfont}%
  \ifx\svgwidth\undefined%
    \setlength{\unitlength}{183.66205891bp}%
    \ifx\svgscale\undefined%
      \relax%
    \else%
      \setlength{\unitlength}{\unitlength * \real{\svgscale}}%
    \fi%
  \else%
    \setlength{\unitlength}{\svgwidth}%
  \fi%
  \global\let\svgwidth\undefined%
  \global\let\svgscale\undefined%
  \makeatother%
  \begin{picture}(1,0.77826771)%
    \lineheight{1}%
    \setlength\tabcolsep{0pt}%
    \put(0.7776146,0.16972158){\color[rgb]{0,0,0}\makebox(0,0)[lt]{\lineheight{1.25}\smash{\begin{tabular}[t]{l}$\scriptstyle T_n$\end{tabular}}}}%
    \put(0.57697216,0.67904447){\color[rgb]{0,0,0}\makebox(0,0)[lt]{\lineheight{1.25}\smash{\begin{tabular}[t]{l}$\scriptstyle T_n$\end{tabular}}}}%
    \put(0,0){\includegraphics[width=\unitlength,page=1]{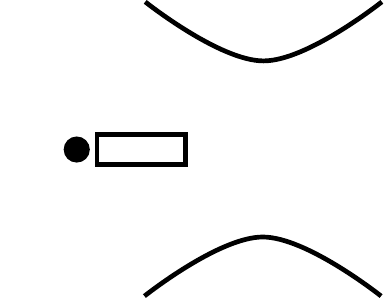}}%
    \put(0.28144691,0.3703639){\color[rgb]{0,0,0}\makebox(0,0)[lt]{\lineheight{1.25}\smash{\begin{tabular}[t]{l}$\scriptstyle n-1$\end{tabular}}}}%
    \put(0,0){\includegraphics[width=\unitlength,page=2]{3.3-Tncross-step3-cupcap.pdf}}%
  \end{picture}%
\endgroup%

}
.
\end{equation}

We begin with the left hand side of (\ref{puncture Steinberg}). After applying isotopies and the identities (\ref{(green strand Steinberg identity)}) from Theorem \ref{Steinberg skein identities} four times in the following way we obtain
\begin{equation}\label{puncture Steinberg crossing}
\centerv { \def\svgscale{0.5}
\begingroup%
  \makeatletter%
  \providecommand\color[2][]{%
    \errmessage{(Inkscape) Color is used for the text in Inkscape, but the package 'color.sty' is not loaded}%
    \renewcommand\color[2][]{}%
  }%
  \providecommand\transparent[1]{%
    \errmessage{(Inkscape) Transparency is used (non-zero) for the text in Inkscape, but the package 'transparent.sty' is not loaded}%
    \renewcommand\transparent[1]{}%
  }%
  \providecommand\rotatebox[2]{#2}%
  \newcommand*\fsize{\dimexpr\f@size pt\relax}%
  \newcommand*\lineheight[1]{\fontsize{\fsize}{#1\fsize}\selectfont}%
  \ifx\svgwidth\undefined%
    \setlength{\unitlength}{288.58614223bp}%
    \ifx\svgscale\undefined%
      \relax%
    \else%
      \setlength{\unitlength}{\unitlength * \real{\svgscale}}%
    \fi%
  \else%
    \setlength{\unitlength}{\svgwidth}%
  \fi%
  \global\let\svgwidth\undefined%
  \global\let\svgscale\undefined%
  \makeatother%
  \begin{picture}(1,0.98626411)%
    \lineheight{1}%
    \setlength\tabcolsep{0pt}%
    \put(0,0){\includegraphics[width=\unitlength,page=1]{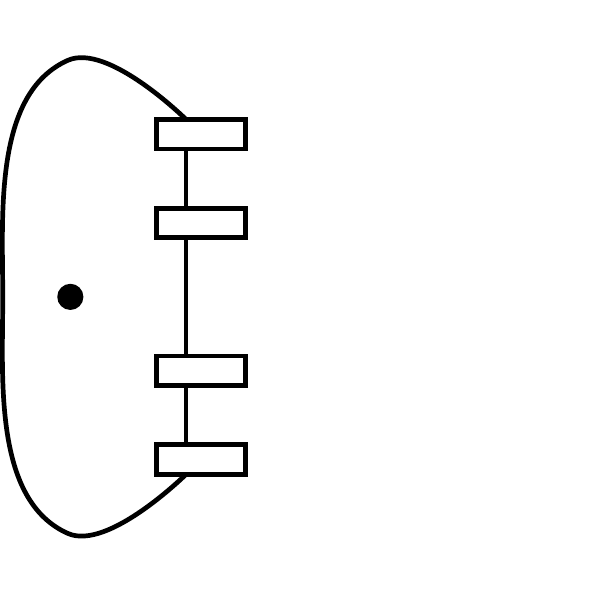}}%
    \put(0.26820472,0.74735373){\color[rgb]{0,0,0}\makebox(0,0)[lt]{\lineheight{1.25}\smash{\begin{tabular}[t]{l}$\scriptstyle n-1$\end{tabular}}}}%
    \put(0.26820472,0.20742367){\color[rgb]{0,0,0}\makebox(0,0)[lt]{\lineheight{1.25}\smash{\begin{tabular}[t]{l}$\scriptstyle n-1$\end{tabular}}}}%
    \put(0.26979011,0.35561659){\color[rgb]{0,0,0}\makebox(0,0)[lt]{\lineheight{1.25}\smash{\begin{tabular}[t]{l}$\scriptstyle n-1$\end{tabular}}}}%
    \put(0.26979011,0.60117981){\color[rgb]{0,0,0}\makebox(0,0)[lt]{\lineheight{1.25}\smash{\begin{tabular}[t]{l}$\scriptstyle n-1$\end{tabular}}}}%
    \put(0,0){\includegraphics[width=\unitlength,page=2]{3.3-Tncross-step4.0.pdf}}%
    \put(0.04503596,0.37367629){\color[rgb]{0,0,0}\rotatebox{90}{\makebox(0,0)[lt]{\lineheight{1.25}\smash{\begin{tabular}[t]{l}$\scriptstyle n-1$\end{tabular}}}}}%
  \end{picture}%
\endgroup%

}
=
\centerv { \def\svgscale{0.5}
\begingroup%
  \makeatletter%
  \providecommand\color[2][]{%
    \errmessage{(Inkscape) Color is used for the text in Inkscape, but the package 'color.sty' is not loaded}%
    \renewcommand\color[2][]{}%
  }%
  \providecommand\transparent[1]{%
    \errmessage{(Inkscape) Transparency is used (non-zero) for the text in Inkscape, but the package 'transparent.sty' is not loaded}%
    \renewcommand\transparent[1]{}%
  }%
  \providecommand\rotatebox[2]{#2}%
  \newcommand*\fsize{\dimexpr\f@size pt\relax}%
  \newcommand*\lineheight[1]{\fontsize{\fsize}{#1\fsize}\selectfont}%
  \ifx\svgwidth\undefined%
    \setlength{\unitlength}{237.00639139bp}%
    \ifx\svgscale\undefined%
      \relax%
    \else%
      \setlength{\unitlength}{\unitlength * \real{\svgscale}}%
    \fi%
  \else%
    \setlength{\unitlength}{\svgwidth}%
  \fi%
  \global\let\svgwidth\undefined%
  \global\let\svgscale\undefined%
  \makeatother%
  \begin{picture}(1,1.20090497)%
    \lineheight{1}%
    \setlength\tabcolsep{0pt}%
    \put(0,0){\includegraphics[width=\unitlength,page=1]{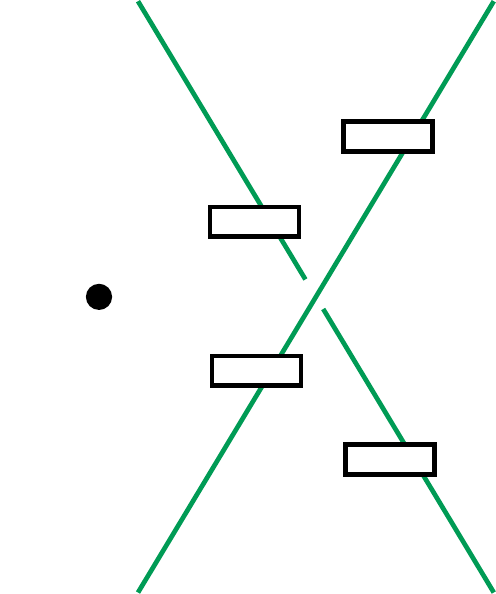}}%
    \put(0.70220785,0.90525384){\color[rgb]{0,0,0}\makebox(0,0)[lt]{\lineheight{1.25}\smash{\begin{tabular}[t]{l}$\scriptstyle 2n-1$\end{tabular}}}}%
    \put(0.70695426,0.25256523){\color[rgb]{0,0,0}\makebox(0,0)[lt]{\lineheight{1.25}\smash{\begin{tabular}[t]{l}$\scriptstyle 2n-1$\end{tabular}}}}%
    \put(0.43978005,0.4330095){\color[rgb]{0,0,0}\makebox(0,0)[lt]{\lineheight{1.25}\smash{\begin{tabular}[t]{l}$\scriptstyle 2n-1$\end{tabular}}}}%
    \put(0.4357098,0.73445689){\color[rgb]{0,0,0}\makebox(0,0)[lt]{\lineheight{1.25}\smash{\begin{tabular}[t]{l}$\scriptstyle 2n-1$\end{tabular}}}}%
    \put(0.07499377,0.54065137){\color[rgb]{0,0,0}\rotatebox{90}{\makebox(0,0)[lt]{\lineheight{1.25}\smash{\begin{tabular}[t]{l}$\scriptstyle n-1$\end{tabular}}}}}%
    \put(0,0){\includegraphics[width=\unitlength,page=2]{3.3-Tncross-step4.4.pdf}}%
    \put(0.43521723,0.51480058){\color[rgb]{0,0,0}\rotatebox{90}{\makebox(0,0)[lt]{\lineheight{1.25}\smash{\begin{tabular}[t]{l}$\scriptstyle n-1$\end{tabular}}}}}%
    \put(0,0){\includegraphics[width=\unitlength,page=3]{3.3-Tncross-step4.4.pdf}}%
    \put(0.55582761,0.85367188){\color[rgb]{0,0,0}\makebox(0,0)[lt]{\lineheight{1.25}\smash{\begin{tabular}[t]{l}$\scriptscriptstyle n-1$\end{tabular}}}}%
    \put(0,0){\includegraphics[width=\unitlength,page=4]{3.3-Tncross-step4.4.pdf}}%
    \put(0.60124286,0.37320846){\color[rgb]{0,0,0}\makebox(0,0)[lt]{\lineheight{1.25}\smash{\begin{tabular}[t]{l}$\scriptscriptstyle n-1$\end{tabular}}}}%
  \end{picture}%
\endgroup%

}.
\end{equation}

On the other hand, we apply the same moves on the right side of (\ref{puncture Steinberg}) to obtain

\begin{equation}\label{puncture steinberg cupcap}
t^{1/2}
\centerv { \def\svgscale{0.5}
\begingroup%
  \makeatletter%
  \providecommand\color[2][]{%
    \errmessage{(Inkscape) Color is used for the text in Inkscape, but the package 'color.sty' is not loaded}%
    \renewcommand\color[2][]{}%
  }%
  \providecommand\transparent[1]{%
    \errmessage{(Inkscape) Transparency is used (non-zero) for the text in Inkscape, but the package 'transparent.sty' is not loaded}%
    \renewcommand\transparent[1]{}%
  }%
  \providecommand\rotatebox[2]{#2}%
  \newcommand*\fsize{\dimexpr\f@size pt\relax}%
  \newcommand*\lineheight[1]{\fontsize{\fsize}{#1\fsize}\selectfont}%
  \ifx\svgwidth\undefined%
    \setlength{\unitlength}{237.00639139bp}%
    \ifx\svgscale\undefined%
      \relax%
    \else%
      \setlength{\unitlength}{\unitlength * \real{\svgscale}}%
    \fi%
  \else%
    \setlength{\unitlength}{\svgwidth}%
  \fi%
  \global\let\svgwidth\undefined%
  \global\let\svgscale\undefined%
  \makeatother%
  \begin{picture}(1,1.20090497)%
    \lineheight{1}%
    \setlength\tabcolsep{0pt}%
    \put(0,0){\includegraphics[width=\unitlength,page=1]{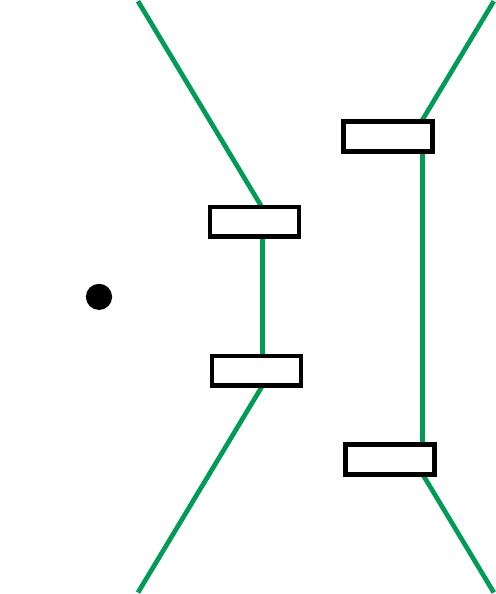}}%
    \put(0.70220787,0.90525385){\color[rgb]{0,0,0}\makebox(0,0)[lt]{\lineheight{1.25}\smash{\begin{tabular}[t]{l}$\scriptstyle 2n-1$\end{tabular}}}}%
    \put(0.70695429,0.25256543){\color[rgb]{0,0,0}\makebox(0,0)[lt]{\lineheight{1.25}\smash{\begin{tabular}[t]{l}$\scriptstyle 2n-1$\end{tabular}}}}%
    \put(0.43978008,0.43300951){\color[rgb]{0,0,0}\makebox(0,0)[lt]{\lineheight{1.25}\smash{\begin{tabular}[t]{l}$\scriptstyle 2n-1$\end{tabular}}}}%
    \put(0.43570982,0.7344569){\color[rgb]{0,0,0}\makebox(0,0)[lt]{\lineheight{1.25}\smash{\begin{tabular}[t]{l}$\scriptstyle 2n-1$\end{tabular}}}}%
    \put(0.07499379,0.54065138){\color[rgb]{0,0,0}\rotatebox{90}{\makebox(0,0)[lt]{\lineheight{1.25}\smash{\begin{tabular}[t]{l}$\scriptstyle n-1$\end{tabular}}}}}%
    \put(0,0){\includegraphics[width=\unitlength,page=2]{3.3-Tncross-step5-vert.pdf}}%
    \put(0.43521726,0.51480059){\color[rgb]{0,0,0}\rotatebox{90}{\makebox(0,0)[lt]{\lineheight{1.25}\smash{\begin{tabular}[t]{l}$\scriptstyle n-1$\end{tabular}}}}}%
    \put(0,0){\includegraphics[width=\unitlength,page=3]{3.3-Tncross-step5-vert.pdf}}%
    \put(0.55680291,0.85381123){\color[rgb]{0,0,0}\makebox(0,0)[lt]{\lineheight{1.25}\smash{\begin{tabular}[t]{l}$\scriptscriptstyle n-1$\end{tabular}}}}%
    \put(0,0){\includegraphics[width=\unitlength,page=4]{3.3-Tncross-step5-vert.pdf}}%
    \put(0.60124288,0.37320847){\color[rgb]{0,0,0}\makebox(0,0)[lt]{\lineheight{1.25}\smash{\begin{tabular}[t]{l}$\scriptscriptstyle n-1$\end{tabular}}}}%
  \end{picture}%
\endgroup%

}
+ t^{-1/2}
\centerv { \def\svgscale{0.5}
\begingroup%
  \makeatletter%
  \providecommand\color[2][]{%
    \errmessage{(Inkscape) Color is used for the text in Inkscape, but the package 'color.sty' is not loaded}%
    \renewcommand\color[2][]{}%
  }%
  \providecommand\transparent[1]{%
    \errmessage{(Inkscape) Transparency is used (non-zero) for the text in Inkscape, but the package 'transparent.sty' is not loaded}%
    \renewcommand\transparent[1]{}%
  }%
  \providecommand\rotatebox[2]{#2}%
  \newcommand*\fsize{\dimexpr\f@size pt\relax}%
  \newcommand*\lineheight[1]{\fontsize{\fsize}{#1\fsize}\selectfont}%
  \ifx\svgwidth\undefined%
    \setlength{\unitlength}{237.00639139bp}%
    \ifx\svgscale\undefined%
      \relax%
    \else%
      \setlength{\unitlength}{\unitlength * \real{\svgscale}}%
    \fi%
  \else%
    \setlength{\unitlength}{\svgwidth}%
  \fi%
  \global\let\svgwidth\undefined%
  \global\let\svgscale\undefined%
  \makeatother%
  \begin{picture}(1,1.20090497)%
    \lineheight{1}%
    \setlength\tabcolsep{0pt}%
    \put(0,0){\includegraphics[width=\unitlength,page=1]{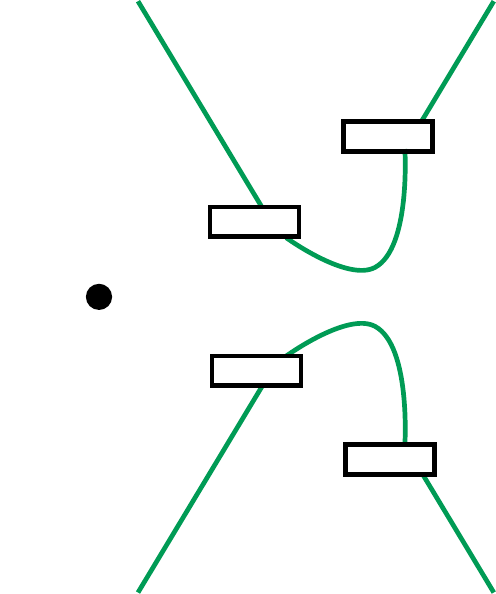}}%
    \put(0.70220769,0.90525385){\color[rgb]{0,0,0}\makebox(0,0)[lt]{\lineheight{1.25}\smash{\begin{tabular}[t]{l}$\scriptstyle 2n-1$\end{tabular}}}}%
    \put(0.7069541,0.25256534){\color[rgb]{0,0,0}\makebox(0,0)[lt]{\lineheight{1.25}\smash{\begin{tabular}[t]{l}$\scriptstyle 2n-1$\end{tabular}}}}%
    \put(0.43977999,0.43300951){\color[rgb]{0,0,0}\makebox(0,0)[lt]{\lineheight{1.25}\smash{\begin{tabular}[t]{l}$\scriptstyle 2n-1$\end{tabular}}}}%
    \put(0.43570973,0.7344569){\color[rgb]{0,0,0}\makebox(0,0)[lt]{\lineheight{1.25}\smash{\begin{tabular}[t]{l}$\scriptstyle 2n-1$\end{tabular}}}}%
    \put(0.07499379,0.54065138){\color[rgb]{0,0,0}\rotatebox{90}{\makebox(0,0)[lt]{\lineheight{1.25}\smash{\begin{tabular}[t]{l}$\scriptstyle n-1$\end{tabular}}}}}%
    \put(0,0){\includegraphics[width=\unitlength,page=2]{3.3-Tncross-step5.0-cupcap.pdf}}%
    \put(0.43521717,0.51480059){\color[rgb]{0,0,0}\rotatebox{90}{\makebox(0,0)[lt]{\lineheight{1.25}\smash{\begin{tabular}[t]{l}$\scriptstyle n-1$\end{tabular}}}}}%
    \put(0,0){\includegraphics[width=\unitlength,page=3]{3.3-Tncross-step5.0-cupcap.pdf}}%
    \put(0.55154258,0.85161678){\color[rgb]{0,0,0}\makebox(0,0)[lt]{\lineheight{1.25}\smash{\begin{tabular}[t]{l}$\scriptscriptstyle n-1$\end{tabular}}}}%
    \put(0,0){\includegraphics[width=\unitlength,page=4]{3.3-Tncross-step5.0-cupcap.pdf}}%
    \put(0.61891921,0.3757337){\color[rgb]{0,0,0}\makebox(0,0)[lt]{\lineheight{1.25}\smash{\begin{tabular}[t]{l}$\scriptscriptstyle n-1$\end{tabular}}}}%
  \end{picture}%
\endgroup%

}.
\end{equation}

We see the that (\ref{puncture Steinberg crossing}) and (\ref{puncture steinberg cupcap}) only differ in the middle regions of the diagrams, where we apply Lemma \ref{lem:ncross} to see that the desired identity (\ref{puncture Steinberg}) holds.

\end{proof}

\section{Thick Jones-Wenzl elements and $JW_{\hat{k}}$}\label{thick JW}

We begin by introducing local notation for elements $\widetilde{JW_{k}}$ which we call thick Jones-Wenzl elements in the skein module. They are meant to be thought of as elements obtained by applying the Frobenius to a neighborhood containing the projector $JW_k.$ Although $Fr(V_k^{(t)})$ is not a tilting module and thus does not exist as an object in the idempotent completion of $\text{TL},$ our elements exist in skein modules. We then use the thick Jones-Wenzl elements $\widetilde{JW_{k}}$ to give new versions of formulas for the elements $JW_{n-1+kn}$ that have appeared in \cite{GW93,MS22,STWZ23}.

Recall the notation from Equation (\ref{green strands}) involving green strands.

\begin{notation}
We will use the following notation to depict the result of starting with the element $JW_k$ at parameter $t$ and then replacing every strand with a green strand. Similar to the previous notation in (\ref{green strands}) regarding green strands, we assume that the green strands leaving $\widetilde{JW_k}$ end at either a standard JW projector or at a thick Jones-Wenzl element. The $\widetilde{JW_k}$ elements satisfy analogous properties, involving green strands, as ordinary $JW$ projectors. In particular they satisfy analogues of the properties of axioms \ref{axiom 1}, \ref{axiom 2} and the absorption property (\ref{absorption}) and equation (\ref{JW trace}). The following depicts $\widetilde{JW_3}.$

\begin{equation*}
\widetilde{JW_3}=
\begin{tikzpicture}[baseline=6ex]
\draw [ultra thick, color=ForestGreen] (1.4,0)--(1.4,2);
\draw [ultra thick, color=ForestGreen] (1.1,0)--(1.1,2);
\draw [ultra thick, color=ForestGreen] (1.7,0)--(1.7,2);
\node [draw=ForestGreen, ellipse, minimum width=.75cm, ultra thick, fill=white] (node) at (1.4,1) {$3$};
\end{tikzpicture}
\end{equation*}
\end{notation}

\begin{example}
The following examples give a sample of the results of connecting up the strands of $\widetilde{JW}_2$ in two different ways.
\begin{equation*}
\begin{tikzpicture}[baseline=6ex]
\draw [ultra thick, color=ForestGreen] (.3,.25)--(.3,1.75);
\draw [ultra thick, color=ForestGreen] (-.3,.25)--(-.3,1.75);
\node [draw, minimum width=.5cm, ultra thick, fill = white] (node) at (.3,1.75) {};
\node [draw, minimum width=.5cm, ultra thick, fill = white] (node) at (.3,.25) {};
\node [draw, minimum width=.5cm, ultra thick, fill = white] (node) at (-.3,1.75) {};
\node [draw, minimum width=.5cm, ultra thick, fill = white] (node) at (-.3,.25) {};
\node [draw=ForestGreen, ellipse, minimum width=.75cm, ultra thick, fill=white] (node) at (0,1) {$2$};
\end{tikzpicture}=
\begin{tikzpicture}[baseline=6ex]
\draw [ultra thick] (.3,.25)--(.3,1.75);
\draw [ultra thick] (-.3,.25)--(-.3,1.75);
\node [draw, minimum width=.5cm, ultra thick, fill = white] (node) at (.3,1.75) {};
\node [draw, minimum width=.5cm, ultra thick, fill = white] (node) at (.3,.25) {};
\node [draw, minimum width=.5cm, ultra thick, fill = white] (node) at (-.3,1.75) {};
\node [draw, minimum width=.5cm, ultra thick, fill = white] (node) at (-.3,.25) {};
\node (node) at (-.3,1) [right] {$n$};
\node (node) at (.3,1) [right] {$n$};
\end{tikzpicture}+ \frac{1}{[2]_t}
\begin{tikzpicture}[baseline=6ex]
\draw [ultra thick] (-.3,1.75)--(-.3,1.5);
\draw [ultra thick] (-.3,1.5) arc (180:360:.3);
\draw [ultra thick] (.3,1.5)--(.3,1.75);
\draw [ultra thick] (-.3,.25)--(-.3,.5);
\draw [ultra thick] (-.3,.5) arc (180:0:.3);
\draw [ultra thick] (.3,.5)--(.3,.25);
\node [draw, minimum width=.5cm, ultra thick, fill = white] (node) at (.3,1.75) {};
\node [draw, minimum width=.5cm, ultra thick, fill = white] (node) at (.3,.25) {};
\node [draw, minimum width=.5cm, ultra thick, fill = white] (node) at (-.3,1.75) {};
\node [draw, minimum width=.5cm, ultra thick, fill = white] (node) at (-.3,.25) {};
\node (node) at (.3,1.4) [right] {$n$};
\node (node) at (.3,.6) [right] {$n$};
\end{tikzpicture}.
\end{equation*}

\begin{equation*}
\begin{tikzpicture}[baseline=6ex]
\draw [ultra thick, color=ForestGreen] (.3,.25)--(.3,1.75);
\draw [ultra thick, color=ForestGreen] (-.3,.25)--(-.3,1.75);
\draw [ultra thick, color=ForestGreen] (.3,1.75) arc (180:0:.25);
\draw [ultra thick, color=ForestGreen] (.8,1.75)--(.8,.25);
\draw [ultra thick, color=ForestGreen] (.3,.25) arc (180:360:.25);
\node [draw, minimum width=.5cm, ultra thick, fill = white] (node) at (-.3,1.75) {};
\node [draw, minimum width=.5cm, ultra thick, fill = white] (node) at (-.3,.25) {};
\node [draw=ForestGreen, ellipse, minimum width=.75cm, ultra thick, fill=white] (node) at (0,1) {$2$};
\end{tikzpicture}=
\begin{tikzpicture}[baseline=6ex]
\draw [ultra thick] (.3,.25)--(.3,1.75);
\draw [ultra thick] (-.3,.25)--(-.3,1.75);
\draw [ultra thick] (.3,1.75) arc (180:0:.25);
\draw [ultra thick] (.8,1.75)--(.8,.25);
\draw [ultra thick] (.3,.25) arc (180:360:.25);
\node [draw, minimum width=.5cm, ultra thick, fill = white] (node) at (-.3,1.75) {};
\node [draw, minimum width=.5cm, ultra thick, fill = white] (node) at (-.3,.25) {};
\node (node) at (-.3,1) [right] {$n$};
\node (node) at (.8,1.75) [right] {$T_n$};
\end{tikzpicture}+\frac{1}{[2]_t}
\begin{tikzpicture}[baseline=6ex]
\draw [ultra thick] (-.3,1.75)--(-.3,1.5);
\draw [ultra thick] (-.3,1.5) arc (180:360:.3);
\draw [ultra thick] (.3,1.75)--(.3,1.5);
\draw [ultra thick] (-.3,.25)--(-.3,.5);
\draw [ultra thick] (-.3,.5) arc (180:0:.3);
\draw [ultra thick] (.3,.25)--(.3,.5);
\draw [ultra thick] (.3,1.75) arc (180:0:.25);
\draw [ultra thick] (.8,1.75)--(.8,.25);
\draw [ultra thick] (.3,.25) arc (180:360:.25);
\node [draw, minimum width=.5cm, ultra thick, fill = white] (node) at (-.3,1.75) {};
\node [draw, minimum width=.5cm, ultra thick, fill = white] (node) at (-.3,.25) {};
\node (node) at (-.3,1) [right] {$n$};
\end{tikzpicture}.
\end{equation*}    
\end{example}

\begin{theorem}
Given $k \geq 1,$ let $\widehat{k}=n-1+kn.$ When $q$ is a root of unity and $n$ is the smallest positive integer such that $q^n \in \{-1,1\},$ the Jones-Wenzl projector $JW_{\widehat{k}}$ is given by

\begin{equation}\label{JWhat}
\begin{tikzpicture}[baseline=6ex]
\draw [ultra thick] (0,0)--(0,2);
\node [draw, minimum width=.75cm, ultra thick, fill = white] (node) at (0,1) {$\widehat{k}$};
\end{tikzpicture}=
\begin{tikzpicture}[baseline=9ex]
\draw[ultra thick] (-.25,-.5)--(-.25,3.5);
\draw[ultra thick, color=ForestGreen] (-.25,1)--(1.4,1.5);
\draw[ultra thick, color=ForestGreen] (-.25,.75)--(1.4,1.5);
\draw[ultra thick, color=ForestGreen] (-.25,2.25)--(1.4,1.5);
\draw[ultra thick, color=ForestGreen] (-.25,2)--(1.4,1.5);
\draw[ultra thick] (.25,2.5)--(1.4,3);
\draw[ultra thick] (.25,.5)--(1.4,0);
\draw[ultra thick, color=ForestGreen] (1.4,0)--(1.4,3);
\draw[ultra thick] (1.4,3)--(1.4,3.5);
\draw[ultra thick] (1.4,0)--(1.4,-.5);
\node[draw, minimum width=.75cm, ultra thick, fill = white] (node) at (0,.75) {$\widehat{k-1}$};
\node[draw, minimum width=.75cm, ultra thick, fill = white] (node) at (0,2.25) {$\widehat{k-1}$};
\node [draw=ForestGreen, ellipse, minimum width=.75cm, ultra thick, fill=white] (node) at (1.4,1.5) {$k$};
\node[draw, minimum width=.75cm, ultra thick, fill = white] (node) at (1.4,3) {$2n-1$};
\node[draw, minimum width=.75cm, ultra thick, fill = white] (node) at (1.4,0) {$2n-1$};
\end{tikzpicture}.
\end{equation}
\end{theorem}

\begin{proof}
We proceed by induction on $k,$ where the base case $k=1$ corresponds to $JW_{2n-1}$. We now assume that the formula for $JW_{\widehat{k-1}}$ satisfies the axioms \ref{axiom 1} and \ref{axiom 2}. First we check that the coefficient of $\text{Id}_{\widehat{k}}$ is $1$ in the expression of $JW_{\widehat{k}}.$ We do this by ignoring any term with a cap, beginning with the bottom $JW_{2n-1}$ and working our way up the diagram and using the fact that the coefficient of the identity term is $1$ for each element appearing in the diagram.

Before checking that $JW_{\widehat{k}}$ is uncappable we observe the following identity:

\begin{equation}\label{trace green strand}
\begin{tikzpicture}[baseline=6ex]
\draw [ultra thick] (-.25,0)--(-.25,2);
\draw [ultra thick, color=ForestGreen] (.25,1)--(.25,1.75);
\draw [ultra thick, color=ForestGreen] (.25,1)--(.25,.25);
\draw [ultra thick,color=ForestGreen] (.25,1.75) arc (180:0:.25);
\draw [ultra thick,color=ForestGreen] (.75,1.75)--(.75,.25);
\draw [ultra thick, color=ForestGreen] (.25,.25) arc (180:360:.25);
\node [draw, minimum width=.75cm, ultra thick, fill = white] (node) at (0,1) {$\widehat{k-1}$};
\end{tikzpicture}\overset{(\ref{JWhat})}{=}
\begin{tikzpicture}[baseline=9ex]
\draw[ultra thick] (-.25,-.5)--(-.25,3.5);
\draw[ultra thick, color=ForestGreen] (-.25,1)--(1.4,1.5);
\draw[ultra thick, color=ForestGreen] (-.25,.75)--(1.4,1.5);
\draw[ultra thick, color=ForestGreen] (-.25,2.25)--(1.4,1.5);
\draw[ultra thick, color=ForestGreen] (-.25,2)--(1.4,1.5);
\draw[ultra thick] (.25,2.5)--(1.4,3);
\draw[ultra thick] (.25,.5)--(1.4,0);
\draw[ultra thick, color=ForestGreen] (1.4,0)--(1.4,3);
\draw[ultra thick] (1.2,3)--(1.2,3.5);
\draw[ultra thick] (1.2,0)--(1.2,-.5);
\draw [ultra thick,color=ForestGreen] (1.8,3.25) arc (180:0:.25);
\draw [ultra thick, color=ForestGreen] (2.3,3.25)--(2.3,-.25);
\draw [ultra thick, color=ForestGreen] (1.8,-.25) arc (180:360:.25);
\node[draw, minimum width=.75cm, ultra thick, fill = white] (node) at (0,.75) {$\widehat{k-2}$};
\node[draw, minimum width=.75cm, ultra thick, fill = white] (node) at (0,2.25) {$\widehat{k-2}$};
\node [draw=ForestGreen, ellipse, minimum width=.75cm, ultra thick, fill=white] (node) at (1.4,1.5) {$k-1$};
\node[draw, minimum width=.75cm, ultra thick, fill = white] (node) at (1.4,3) {$2n-1$};
\node[draw, minimum width=.75cm, ultra thick, fill = white] (node) at (1.4,0) {$2n-1$};
\end{tikzpicture}\overset{(\ref{(green strand Steinberg identity)})}{=}
\begin{tikzpicture}[baseline=9ex]
\draw[ultra thick] (-.25,-.5)--(-.25,3.5);
\draw[ultra thick, color=ForestGreen] (-.25,1)--(1.4,1.5);
\draw[ultra thick, color=ForestGreen] (-.25,.75)--(1.4,1.5);
\draw[ultra thick, color=ForestGreen] (-.25,2.25)--(1.4,1.5);
\draw[ultra thick, color=ForestGreen] (-.25,2)--(1.4,1.5);
\draw [ultra thick,color=ForestGreen] (1.8,1.75) arc (180:0:.25);
\draw [ultra thick, color=ForestGreen] (2.3,1.75)--(2.3,1.25);
\draw [ultra thick, color=ForestGreen] (1.8,1.25) arc (180:360:.25);
\node[draw, minimum width=.75cm, ultra thick, fill = white] (node) at (0,.75) {$\widehat{k-2}$};
\node[draw, minimum width=.75cm, ultra thick, fill = white] (node) at (0,2.25) {$\widehat{k-2}$};
\node [draw=ForestGreen, ellipse, minimum width=.75cm, ultra thick, fill=white] (node) at (1.4,1.5) {$k-1$};
\end{tikzpicture}\overset{(\ref{JW trace})}{=}\frac{-[k]_t}{[k-1]_t}
\begin{tikzpicture}[baseline=6ex]
\draw [ultra thick] (0,0)--(0,2);
\node [draw, minimum width=.75cm, ultra thick, fill = white] (node) at (0,1) {$\widehat{k-2}$};
\end{tikzpicture}.
\end{equation}

Next, we check that $JW_{\widehat{k}}$ is uncappable. We have assumed that $JW_{\widehat{k-1}}$ is uncappable and we know $JW_{2n-1}$ is uncappable. Thus, the only spot that remains to check is between strands $kn-n$ and $kn-n+1.$

We insert the cap and then replace the top right $JW_{2n-1}$ with the expression (\ref{JW_2n-1}). Since each green strand is $n$ parallel copies of a strand ultimately connected to a $JW$ projector, only the $A_{n-1}'$ term survives and we obtain the following:

\begin{equation*}
\begin{tikzpicture}[baseline=9ex]
\draw[ultra thick] (-.25,-.5)--(-.25,3.5);
\draw[ultra thick, color=ForestGreen] (-.25,1)--(1.4,1.5);
\draw[ultra thick, color=ForestGreen] (-.25,.75)--(1.4,1.5);
\draw[ultra thick, color=ForestGreen] (-.25,2.25)--(1.4,1.5);
\draw[ultra thick, color=ForestGreen] (-.25,2)--(1.4,1.5);
\draw[ultra thick] (.25,2.5)--(1.4,3);
\draw[ultra thick] (.25,.5)--(1.4,0);
\draw[ultra thick, color=ForestGreen] (1.4,0)--(1.4,3);
\draw[ultra thick] (1.4,3)--(1.4,3.5);
\draw[ultra thick] (1.4,0)--(1.4,-.5);
\draw[ultra thick] (1.2,3.25) arc (0:90:.25);
\draw[ultra thick] (.95,3.5)--(.5,3.5);
\draw[ultra thick] (.5,3.5) arc (90:180:.25);
\draw[ultra thick] (.25,2.25)--(.25,3.25);
\node (node) at (.75,3.5) [above] {$1$};
\node[draw, minimum width=.75cm, ultra thick, fill = white] (node) at (0,.75) {$\widehat{k-1}$};
\node[draw, minimum width=.75cm, ultra thick, fill = white] (node) at (0,2.25) {$\widehat{k-1}$};
\node [draw=ForestGreen, ellipse, minimum width=.75cm, ultra thick, fill=white] (node) at (1.4,1.5) {$k$};
\node[draw, minimum width=.75cm, ultra thick, fill = white] (node) at (1.4,3) {$2n-1$};
\node[draw, minimum width=.75cm, ultra thick, fill = white] (node) at (1.4,0) {$2n-1$};
\end{tikzpicture}=(-1)^{n-1}
\begin{tikzpicture}[baseline=9ex]
\draw[ultra thick] (-.25,-.5)--(-.25,3.5);
\draw[ultra thick, color=ForestGreen] (-.25,1)--(1.4,1.5);
\draw[ultra thick, color=ForestGreen] (-.25,.75)--(1.4,1.5);
\draw[ultra thick, color=ForestGreen] (-.25,2.25)--(1.4,1.5);
\draw[ultra thick, color=ForestGreen] (-.25,2)--(1.4,1.5);
\draw[ultra thick] (.25,.5)--(1.4,0);
\draw[ultra thick, color=ForestGreen] (1.4,0)--(1.4,1.5);
\draw[ultra thick] (1.4,3)--(1.4,3.5);
\draw[ultra thick] (1.4,0)--(1.4,-.5);
\draw[ultra thick, color=ForestGreen] (1.4,1.5)--(.5,2.75);
\draw[ultra thick, color=ForestGreen] (.5,2.75) arc (90:180:.25);
\draw[ultra thick] (1.4,2.75) arc (0:-90:.25);
\draw[ultra thick] (1.15,2.5)--(.9,2.5);
\draw[ultra thick] (.9,2.5) arc (-90:-180:.25);
\draw[ultra thick] (.65,2.75)--(.65,3.5);
\node[draw, minimum width=.75cm, ultra thick, fill = white] (node) at (0,.75) {$\widehat{k-1}$};
\node[draw, minimum width=.75cm, ultra thick, fill = white] (node) at (0,2.25) {$\widehat{k-1}$};
\node [draw=ForestGreen, ellipse, minimum width=.75cm, ultra thick, fill=white] (node) at (1.4,1.5) {$k$};
\node[draw, minimum width=.75cm, ultra thick, fill = white] (node) at (1.4,3) {$n-1$};
\node[draw, minimum width=.75cm, ultra thick, fill = white] (node) at (1.4,0) {$2n-1$};
\end{tikzpicture}.
\end{equation*}

We then apply the analogue of $JW$ recursion (\ref{JW recursion}) to $\widetilde{JW_{k}}$ and absorb the resulting $\widetilde{JW_{k-1}}$ elements into the $JW_{\widehat{k-1}}$ elements to yield

\begin{equation*}
(-1)^{n-1}\left(
\begin{tikzpicture}[baseline=9ex]
\draw[ultra thick] (-.25,-.5)--(-.25,3.5);
\draw[ultra thick] (.25,.5)--(1.4,0);
\draw[ultra thick, color=ForestGreen] (1.4,0)--(1.4,1.5);
\draw[ultra thick] (1.4,3)--(1.4,3.5);
\draw[ultra thick] (1.4,0)--(1.4,-.5);
\draw[ultra thick, color=ForestGreen] (1.4,1.5)--(.5,2.75);
\draw[ultra thick, color=ForestGreen] (.5,2.75) arc (90:180:.25);
\draw[ultra thick] (1.4,2.75) arc (0:-90:.25);
\draw[ultra thick] (1.15,2.5)--(.9,2.5);
\draw[ultra thick] (.9,2.5) arc (-90:-180:.25);
\draw[ultra thick] (.65,2.75)--(.65,3.5);
\node[draw, minimum width=.75cm, ultra thick, fill = white] (node) at (0,.75) {$\widehat{k-1}$};
\node[draw, minimum width=.75cm, ultra thick, fill = white] (node) at (0,2.25) {$\widehat{k-1}$};
\node[draw, minimum width=.75cm, ultra thick, fill = white] (node) at (1.4,3) {$n-1$};
\node[draw, minimum width=.75cm, ultra thick, fill = white] (node) at (1.4,0) {$2n-1$};
\end{tikzpicture}+\frac{[k-1]_t}{[k]_t}
\begin{tikzpicture}[baseline=9ex]
\draw[ultra thick] (-.25,-.5)--(-.25,3.5);
\draw[ultra thick] (.25,.5)--(1.4,0);
\draw[ultra thick, color=ForestGreen] (1.4,0)--(1.4,1);
\draw[ultra thick] (1.4,3)--(1.4,3.5);
\draw[ultra thick] (1.4,0)--(1.4,-.5);
\draw[ultra thick, color=ForestGreen] (.25,2.25) arc (180:0:.25);
\draw[ultra thick, color=ForestGreen] (.75,2.25)--(.75,1.7);
\draw[ultra thick, color=ForestGreen] (.75,1.7) arc (0:-180:.25);
\draw[ultra thick] (1.4,2.75) arc (0:-90:.25);
\draw[ultra thick] (1.15,2.5)--(.9,2.5);
\draw[ultra thick] (.9,2.5) arc (-90:-180:.25);
\draw[ultra thick] (.65,2.75)--(.65,3.5);
\draw[ultra thick, color=ForestGreen] (.25,1) arc (180:90:.25);
\draw[ultra thick, color=ForestGreen] (.5,1.25)--(1.15,1.25);
\draw[ultra thick, color=ForestGreen] (1.15,1.25) arc (90:0:.25);
\node[draw, minimum width=.75cm, ultra thick, fill = white] (node) at (0,.75) {$\widehat{k-1}$};
\node[draw, minimum width=.75cm, ultra thick, fill = white] (node) at (0,2) {$\widehat{k-1}$};
\node[draw, minimum width=.75cm, ultra thick, fill = white] (node) at (1.4,3) {$n-1$};
\node[draw, minimum width=.75cm, ultra thick, fill = white] (node) at (1.4,0) {$2n-1$};
\end{tikzpicture}
\right)=0.
\end{equation*}

To obtain zero we observe that the two elements simplify to the same element with opposite sign after applying equation (\ref{trace green strand}) and then absorbing $JW$ projectors. We conclude that $JW_{\widehat{k}}$ satisfies both axioms \ref{axiom 1} and \ref{axiom 2}.
\end{proof}

\section{Other Steinberg skein identities}

We conclude by showing skein identities relating the thick Jones-Wenzl elements $\widetilde{JW_k}$ to $JW_{n-1}$ and $JW_{\widehat{k}}$.

\begin{theorem}\label{other Steinberg skein identities}
Suppose $q$ is a root of unity and $n$ is the smallest positive integer such that $q^n \in \{-1,1\}.$ Then the following skein identites hold in the skein module.

\begin{equation}\label{further skein identities}
\begin{tikzpicture}[baseline=6ex]
\draw [ultra thick] (0,0)--(0,2);
\draw [ultra thick, color=ForestGreen] (1.4,0)--(1.4,2);
\draw [ultra thick, color=ForestGreen] (1.1,0)--(1.1,2);
\draw [ultra thick, color=ForestGreen] (1.7,0)--(1.7,2);
\node [draw, minimum width=.75cm, ultra thick, fill = white] (node) at (0,1) {$n-1$};
\node [draw=ForestGreen, ellipse, minimum width=.75cm, ultra thick, fill=white] (node) at (1.4,1) {$k$};
\end{tikzpicture}=
\begin{tikzpicture}[baseline=6ex]
\draw [ultra thick] (-.25,0)--(-.25,2);
\draw [ultra thick, color=ForestGreen] (.05,0)--(.05,2);
\draw [ultra thick, color=ForestGreen] (.35,0)--(.35,2);
\draw [ultra thick, color=ForestGreen] (.65,0)--(.65,2);
\node [draw, minimum width=1.4cm, ultra thick, fill = white] (node) at (0,1) {$\widehat{k}$};
\node (node) at (-.25,2) [above] {$n-1$};
\node (node) at (-.25,0) [below] {$n-1$};
\end{tikzpicture}.
\end{equation}
\end{theorem}

The $k=1$ case was already shown in Theorem \ref{Steinberg skein identities}, which contained essentially two different identities. Similarly, for each $k$ the identities of Theorem \ref{other Steinberg skein identities} actually describe many skein identities in $\Sq(M).$

\begin{proof}
We proceed by induction. The base case of $k=1$ was shown in Theorem \ref{Steinberg skein identities}. We inductively assume that (\ref{further skein identities}) holds for the $k-1$ case. To complete the proof we use the absorption property on $\widetilde{JW_k}$ and then use Equation (\ref{further skein identities}) on $\widetilde{JW_{k-1}}$ in the following way.

\begin{equation}
\centerv { \def\svgscale{0.6}
\begingroup%
  \makeatletter%
  \providecommand\color[2][]{%
    \errmessage{(Inkscape) Color is used for the text in Inkscape, but the package 'color.sty' is not loaded}%
    \renewcommand\color[2][]{}%
  }%
  \providecommand\transparent[1]{%
    \errmessage{(Inkscape) Transparency is used (non-zero) for the text in Inkscape, but the package 'transparent.sty' is not loaded}%
    \renewcommand\transparent[1]{}%
  }%
  \providecommand\rotatebox[2]{#2}%
  \newcommand*\fsize{\dimexpr\f@size pt\relax}%
  \newcommand*\lineheight[1]{\fontsize{\fsize}{#1\fsize}\selectfont}%
  \ifx\svgwidth\undefined%
    \setlength{\unitlength}{129.80940331bp}%
    \ifx\svgscale\undefined%
      \relax%
    \else%
      \setlength{\unitlength}{\unitlength * \real{\svgscale}}%
    \fi%
  \else%
    \setlength{\unitlength}{\svgwidth}%
  \fi%
  \global\let\svgwidth\undefined%
  \global\let\svgscale\undefined%
  \makeatother%
  \begin{picture}(1,1.20103412)%
    \lineheight{1}%
    \setlength\tabcolsep{0pt}%
    \put(0,0){\includegraphics[width=\unitlength,page=1]{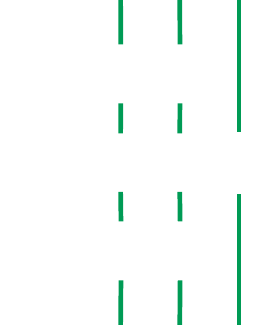}}%
    \put(0.02697905,0.90528112){\color[rgb]{0,0,0}\makebox(0,0)[lt]{\lineheight{1.25}\smash{\begin{tabular}[t]{l}$\scriptscriptstyle n-1$\end{tabular}}}}%
    \put(0,0){\includegraphics[width=\unitlength,page=2]{4.2-hS-step1.pdf}}%
    \put(0.02697905,0.25017172){\color[rgb]{0,0,0}\makebox(0,0)[lt]{\lineheight{1.25}\smash{\begin{tabular}[t]{l}$\scriptscriptstyle n-1$\end{tabular}}}}%
    \put(0,0){\includegraphics[width=\unitlength,page=3]{4.2-hS-step1.pdf}}%
    \put(0.5451399,0.55389174){\color[rgb]{0,0,0}\makebox(0,0)[lt]{\lineheight{1.25}\smash{\begin{tabular}[t]{l}$\scriptstyle k$\end{tabular}}}}%
    \put(0,0){\includegraphics[width=\unitlength,page=4]{4.2-hS-step1.pdf}}%
    \put(0.43357814,0.88727209){\color[rgb]{0,0,0}\makebox(0,0)[lt]{\lineheight{1.25}\smash{\begin{tabular}[t]{l}$\scriptstyle k-1$\end{tabular}}}}%
    \put(0,0){\includegraphics[width=\unitlength,page=5]{4.2-hS-step1.pdf}}%
    \put(0.43416167,0.23283136){\color[rgb]{0,0,0}\makebox(0,0)[lt]{\lineheight{1.25}\smash{\begin{tabular}[t]{l}$\scriptstyle k-1$\end{tabular}}}}%
    \put(0,0){\includegraphics[width=\unitlength,page=6]{4.2-hS-step1.pdf}}%
    \put(0.46743304,1.11366422){\color[rgb]{0,0,0}\makebox(0,0)[lt]{\lineheight{1.25}\smash{\begin{tabular}[t]{l}\textcolor{ForestGreen}{$\scriptstyle \ldots$}\end{tabular}}}}%
    \put(0.46567863,0.7622782){\color[rgb]{0,0,0}\makebox(0,0)[lt]{\lineheight{1.25}\smash{\begin{tabular}[t]{l}\textcolor{ForestGreen}{$\scriptstyle \ldots$}\end{tabular}}}}%
    \put(0.46312298,0.43430158){\color[rgb]{0,0,0}\makebox(0,0)[lt]{\lineheight{1.25}\smash{\begin{tabular}[t]{l}\textcolor{ForestGreen}{$\scriptstyle \ldots$}\end{tabular}}}}%
    \put(0.46739393,0.09945197){\color[rgb]{0,0,0}\makebox(0,0)[lt]{\lineheight{1.25}\smash{\begin{tabular}[t]{l}\textcolor{ForestGreen}{$\scriptstyle \ldots$}\end{tabular}}}}%
  \end{picture}%
\endgroup%

}\overset{(\ref{further skein identities})}{=}
\centerv { \def\svgscale{0.6}
\begingroup%
  \makeatletter%
  \providecommand\color[2][]{%
    \errmessage{(Inkscape) Color is used for the text in Inkscape, but the package 'color.sty' is not loaded}%
    \renewcommand\color[2][]{}%
  }%
  \providecommand\transparent[1]{%
    \errmessage{(Inkscape) Transparency is used (non-zero) for the text in Inkscape, but the package 'transparent.sty' is not loaded}%
    \renewcommand\transparent[1]{}%
  }%
  \providecommand\rotatebox[2]{#2}%
  \newcommand*\fsize{\dimexpr\f@size pt\relax}%
  \newcommand*\lineheight[1]{\fontsize{\fsize}{#1\fsize}\selectfont}%
  \ifx\svgwidth\undefined%
    \setlength{\unitlength}{115.63617496bp}%
    \ifx\svgscale\undefined%
      \relax%
    \else%
      \setlength{\unitlength}{\unitlength * \real{\svgscale}}%
    \fi%
  \else%
    \setlength{\unitlength}{\svgwidth}%
  \fi%
  \global\let\svgwidth\undefined%
  \global\let\svgscale\undefined%
  \makeatother%
  \begin{picture}(1,1.74045745)%
    \lineheight{1}%
    \setlength\tabcolsep{0pt}%
    \put(0,0){\includegraphics[width=\unitlength,page=1]{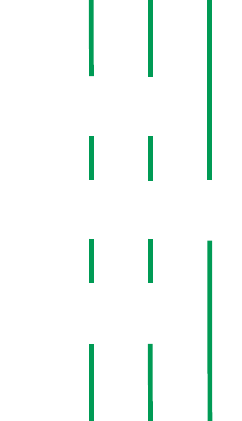}}%
    \put(0.04593196,1.24560571){\color[rgb]{0,0,0}\makebox(0,0)[lt]{\lineheight{1.25}\smash{\begin{tabular}[t]{l}$\scriptstyle \widehat{k-1}$\end{tabular}}}}%
    \put(0,0){\includegraphics[width=\unitlength,page=2]{4.2-hS-step2.pdf}}%
    \put(0.48938886,0.81788867){\color[rgb]{0,0,0}\makebox(0,0)[lt]{\lineheight{1.25}\smash{\begin{tabular}[t]{l}$\scriptstyle k$\end{tabular}}}}%
    \put(0,0){\includegraphics[width=\unitlength,page=3]{4.2-hS-step2.pdf}}%
    \put(0.04593196,0.38763354){\color[rgb]{0,0,0}\makebox(0,0)[lt]{\lineheight{1.25}\smash{\begin{tabular}[t]{l}$\scriptstyle \widehat{k-1}$\end{tabular}}}}%
    \put(0,0){\includegraphics[width=\unitlength,page=4]{4.2-hS-step2.pdf}}%
    \put(0.07966363,0.70818263){\color[rgb]{0,0,0}\transparent{0}\rotatebox{90}{\makebox(0,0)[lt]{\lineheight{1.25}\smash{\begin{tabular}[t]{l}$\scriptstyle n-1$\end{tabular}}}}}%
    \put(0.06977684,1.49770044){\color[rgb]{0,0,0}\transparent{0}\rotatebox{90}{\makebox(0,0)[lt]{\lineheight{1.25}\smash{\begin{tabular}[t]{l}$\scriptstyle n-1$\end{tabular}}}}}%
    \put(0.07947473,0.00100621){\color[rgb]{0,0,0}\transparent{0}\rotatebox{90}{\makebox(0,0)[lt]{\lineheight{1.25}\smash{\begin{tabular}[t]{l}$\scriptstyle n-1$\end{tabular}}}}}%
    \put(0.42324872,1.56354284){\color[rgb]{0,0,0}\makebox(0,0)[lt]{\lineheight{1.25}\smash{\begin{tabular}[t]{l}\textcolor{ForestGreen}{$\scriptstyle \ldots$}\end{tabular}}}}%
    \put(0.42134745,0.20139991){\color[rgb]{0,0,0}\makebox(0,0)[lt]{\lineheight{1.25}\smash{\begin{tabular}[t]{l}\textcolor{ForestGreen}{$\scriptstyle \ldots$}\end{tabular}}}}%
    \put(0.43127837,1.07813835){\color[rgb]{0,0,0}\makebox(0,0)[lt]{\lineheight{1.25}\smash{\begin{tabular}[t]{l}\textcolor{ForestGreen}{$\scriptstyle \ldots$}\end{tabular}}}}%
    \put(0.42617171,0.65644738){\color[rgb]{0,0,0}\makebox(0,0)[lt]{\lineheight{1.25}\smash{\begin{tabular}[t]{l}\textcolor{ForestGreen}{$\scriptstyle \ldots$}\end{tabular}}}}%
  \end{picture}%
\endgroup%

} \overset{(\ref{absorption})}{=} 
\centerv { \def\svgscale{0.6}
\begingroup%
  \makeatletter%
  \providecommand\color[2][]{%
    \errmessage{(Inkscape) Color is used for the text in Inkscape, but the package 'color.sty' is not loaded}%
    \renewcommand\color[2][]{}%
  }%
  \providecommand\transparent[1]{%
    \errmessage{(Inkscape) Transparency is used (non-zero) for the text in Inkscape, but the package 'transparent.sty' is not loaded}%
    \renewcommand\transparent[1]{}%
  }%
  \providecommand\rotatebox[2]{#2}%
  \newcommand*\fsize{\dimexpr\f@size pt\relax}%
  \newcommand*\lineheight[1]{\fontsize{\fsize}{#1\fsize}\selectfont}%
  \ifx\svgwidth\undefined%
    \setlength{\unitlength}{129.80940331bp}%
    \ifx\svgscale\undefined%
      \relax%
    \else%
      \setlength{\unitlength}{\unitlength * \real{\svgscale}}%
    \fi%
  \else%
    \setlength{\unitlength}{\svgwidth}%
  \fi%
  \global\let\svgwidth\undefined%
  \global\let\svgscale\undefined%
  \makeatother%
  \begin{picture}(1,1.85614335)%
    \lineheight{1}%
    \setlength\tabcolsep{0pt}%
    \put(0,0){\includegraphics[width=\unitlength,page=1]{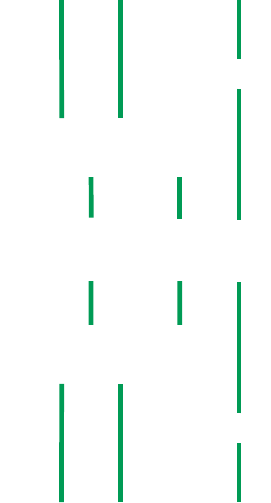}}%
    \put(0.25201516,1.27095372){\color[rgb]{0,0,0}\makebox(0,0)[lt]{\lineheight{1.25}\smash{\begin{tabular}[t]{l}$\scriptstyle \widehat{k-1}$\end{tabular}}}}%
    \put(0,0){\includegraphics[width=\unitlength,page=2]{4.2-hS-step3.pdf}}%
    \put(0.51219599,0.89513019){\color[rgb]{0,0,0}\makebox(0,0)[lt]{\lineheight{1.25}\smash{\begin{tabular}[t]{l}$\scriptstyle k$\end{tabular}}}}%
    \put(0,0){\includegraphics[width=\unitlength,page=3]{4.2-hS-step3.pdf}}%
    \put(0.25201516,0.50665908){\color[rgb]{0,0,0}\makebox(0,0)[lt]{\lineheight{1.25}\smash{\begin{tabular}[t]{l}$\scriptstyle \widehat{k-1}$\end{tabular}}}}%
    \put(0,0){\includegraphics[width=\unitlength,page=4]{4.2-hS-step3.pdf}}%
    \put(0.07963208,0.78371863){\color[rgb]{0,0,0}\transparent{0}\rotatebox{90}{\makebox(0,0)[lt]{\lineheight{1.25}\smash{\begin{tabular}[t]{l}$\scriptstyle n-1$\end{tabular}}}}}%
    \put(0.07100121,1.52548255){\color[rgb]{0,0,0}\transparent{0}\rotatebox{90}{\makebox(0,0)[lt]{\lineheight{1.25}\smash{\begin{tabular}[t]{l}$\scriptstyle n-1$\end{tabular}}}}}%
    \put(0.08100306,0.06557692){\color[rgb]{0,0,0}\transparent{0}\rotatebox{90}{\makebox(0,0)[lt]{\lineheight{1.25}\smash{\begin{tabular}[t]{l}$\scriptstyle n-1$\end{tabular}}}}}%
    \put(0,0){\includegraphics[width=\unitlength,page=5]{4.2-hS-step3.pdf}}%
    \put(0.68357879,1.55798154){\color[rgb]{0,0,0}\makebox(0,0)[lt]{\lineheight{1.25}\smash{\begin{tabular}[t]{l}$\scriptscriptstyle 2n-1$\end{tabular}}}}%
    \put(0,0){\includegraphics[width=\unitlength,page=6]{4.2-hS-step3.pdf}}%
    \put(0.68357879,0.24776258){\color[rgb]{0,0,0}\makebox(0,0)[lt]{\lineheight{1.25}\smash{\begin{tabular}[t]{l}$\scriptscriptstyle 2n-1$\end{tabular}}}}%
    \put(0,0){\includegraphics[width=\unitlength,page=7]{4.2-hS-step3.pdf}}%
    \put(0.26785144,1.63105155){\color[rgb]{0,0,0}\makebox(0,0)[lt]{\lineheight{1.25}\smash{\begin{tabular}[t]{l}\textcolor{ForestGreen}{$\scriptstyle \ldots$}\end{tabular}}}}%
    \put(0,0){\includegraphics[width=\unitlength,page=8]{4.2-hS-step3.pdf}}%
    \put(0.52082253,1.50037257){\color[rgb]{0,0,0}\makebox(0,0)[lt]{\lineheight{1.25}\smash{\begin{tabular}[t]{l}$\scriptscriptstyle 1$\end{tabular}}}}%
    \put(0,0){\includegraphics[width=\unitlength,page=9]{4.2-hS-step3.pdf}}%
    \put(0.27181759,0.19660432){\color[rgb]{0,0,0}\makebox(0,0)[lt]{\lineheight{1.25}\smash{\begin{tabular}[t]{l}\textcolor{ForestGreen}{$\scriptstyle \ldots$}\end{tabular}}}}%
    \put(0.52191195,0.32627989){\color[rgb]{0,0,0}\makebox(0,0)[lt]{\lineheight{1.25}\smash{\begin{tabular}[t]{l}$\scriptscriptstyle 1$\end{tabular}}}}%
    \put(0.43390993,1.11793182){\color[rgb]{0,0,0}\makebox(0,0)[lt]{\lineheight{1.25}\smash{\begin{tabular}[t]{l}\textcolor{ForestGreen}{$\scriptstyle \ldots$}\end{tabular}}}}%
    \put(0.42296545,0.73556188){\color[rgb]{0,0,0}\makebox(0,0)[lt]{\lineheight{1.25}\smash{\begin{tabular}[t]{l}\textcolor{ForestGreen}{$\scriptstyle \ldots$}\end{tabular}}}}%
    \put(0,0){\includegraphics[width=\unitlength,page=10]{4.2-hS-step3.pdf}}%
  \end{picture}%
\endgroup%

}\overset{(\ref{JWhat})}{=} 
\centerv { \def\svgscale{0.6}
\begingroup%
  \makeatletter%
  \providecommand\color[2][]{%
    \errmessage{(Inkscape) Color is used for the text in Inkscape, but the package 'color.sty' is not loaded}%
    \renewcommand\color[2][]{}%
  }%
  \providecommand\transparent[1]{%
    \errmessage{(Inkscape) Transparency is used (non-zero) for the text in Inkscape, but the package 'transparent.sty' is not loaded}%
    \renewcommand\transparent[1]{}%
  }%
  \providecommand\rotatebox[2]{#2}%
  \newcommand*\fsize{\dimexpr\f@size pt\relax}%
  \newcommand*\lineheight[1]{\fontsize{\fsize}{#1\fsize}\selectfont}%
  \ifx\svgwidth\undefined%
    \setlength{\unitlength}{101.46328183bp}%
    \ifx\svgscale\undefined%
      \relax%
    \else%
      \setlength{\unitlength}{\unitlength * \real{\svgscale}}%
    \fi%
  \else%
    \setlength{\unitlength}{\svgwidth}%
  \fi%
  \global\let\svgwidth\undefined%
  \global\let\svgscale\undefined%
  \makeatother%
  \begin{picture}(1,0.97782841)%
    \lineheight{1}%
    \setlength\tabcolsep{0pt}%
    \put(0,0){\includegraphics[width=\unitlength,page=1]{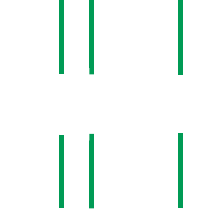}}%
    \put(0.10215888,0.70610403){\color[rgb]{0,0,0}\transparent{0}\rotatebox{90}{\makebox(0,0)[lt]{\lineheight{1.25}\smash{\begin{tabular}[t]{l}$\scriptscriptstyle n-1$\end{tabular}}}}}%
    \put(0.45325563,0.43477746){\color[rgb]{0,0,0}\makebox(0,0)[lt]{\lineheight{1.25}\smash{\begin{tabular}[t]{l}$\scriptstyle \widehat{k}$\end{tabular}}}}%
    \put(0,0){\includegraphics[width=\unitlength,page=2]{4.2-hS-end.pdf}}%
    \put(0.10279235,0.03879359){\color[rgb]{0,0,0}\transparent{0}\rotatebox{90}{\makebox(0,0)[lt]{\lineheight{1.25}\smash{\begin{tabular}[t]{l}$\scriptscriptstyle n-1$\end{tabular}}}}}%
    \put(0.58084265,0.72524552){\color[rgb]{0,0,0}\makebox(0,0)[lt]{\lineheight{1.25}\smash{\begin{tabular}[t]{l}\textcolor{ForestGreen}{$\scriptstyle \ldots$}\end{tabular}}}}%
    \put(0.58466311,0.18383397){\color[rgb]{0,0,0}\makebox(0,0)[lt]{\lineheight{1.25}\smash{\begin{tabular}[t]{l}\textcolor{ForestGreen}{$\scriptstyle \ldots$}\end{tabular}}}}%
  \end{picture}%
\endgroup%

}.
\end{equation}

The second equality follows from pulling $JW_{n-1}$ out of the rightmost parts of $JW_{\widehat{k-1}}$ and the third equality follows from the $k=1$ case of Equation (\ref{further skein identities}).
\end{proof}

\bibliographystyle{amsalpha}
\bibliography{references}
\end{document}